\documentclass[a4paper,12pt]{amsart}
\usepackage{graphicx}
\usepackage{amsmath,amssymb,amsthm,amsfonts}
\usepackage{amssymb}
\usepackage[active]{srcltx}
\usepackage{todonotes}
\usepackage[hyperpageref]{backref}
\usepackage[colorlinks=true,urlcolor=blue,citecolor=red,linkcolor=blue,linktocpage,pdfpagelabels,bookmarksnumbered,bookmarksopen]{hyperref}
\usepackage{esint}
\usepackage{latexsym}
\newtheorem{thm}{Theorem}[section]

\newtheorem{lem}[thm]{Lemma}
\newtheorem{prop}[thm]{Proposition}
\newtheorem{rem}[thm]{Remark}

\newcommand{\R}{{\mathbb{R}}}
\newcommand{\N}{{\mathbb{N}}}

 \numberwithin{equation}{section}

\begin{document}

\parindent 0pc
\parskip 6pt
\overfullrule=0pt

\title[Asymptotic behaviour of solutions]{Asymptotic behaviour of solutions to the anisotropic  doubly critical equation}

\author{Francesco Esposito$^{*}$, Luigi Montoro$^{*}$, Berardino Sciunzi$^{*}$, Domenico Vuono$^{*}$}
\address{$^{*}$Dipartimento di Matematica e Informatica, UNICAL, Ponte Pietro  Bucci 31B, 87036 Arcavacata di Rende, Cosenza, Italy}
\email{francesco.esposito@unical.it, luigi.montoro@unical.it, sciunzi@mat.unical.it, domenico.vuono@unical.it}
\thanks{The authors are members of INdAM. 
The authors are partially supported by PRIN project 2017JPCAPN (Italy): 
{\em Qualitative and quantitative aspects of nonlinear PDEs.} L. Montoro is partially 
supported by  Agencia Estatal de Investigación (Spain): {\em project PDI2019-110712GB-100}.
}

\keywords{Doubly critical equation, asymptotic behaviour, Finsler anisotropic operator,  quasilinear equations}

\subjclass[2020]{35J62, 35B33, 35B40, 35A01}


\maketitle

\date{\today}

\begin{abstract}
The aim of this paper is to deal with the anisotropic  doubly critical equation 
$$-\Delta_p^H u - \frac{\gamma}{[H^\circ(x)]^p} u^{p-1} = u^{p^*-1} \qquad \text{in } \R^N,$$
where $H$ is in some cases called  Finsler norm, $H^\circ$ is the dual norm, $1<p<N$, $0 \leq \gamma< \left((N-p)/p\right)^p$ and $p^*=Np/(N-p)$. In particular, we provide a complete 
 asymptotic analysis of $u \in \mathcal{D}^{1,p}(\R^N)$ near the origin and at infinity, showing that this solution has the same features of its euclidean counterpart. Some of the techniques used in the proofs are new even in the Euclidean framework.
\end{abstract}

\section{Introduction and main results}

This work is devoted to the study of the following anisotropic doubly critical problem
\begin{equation}\label{eq:ProbDoublyCritical} \tag{$\mathcal{P}_H$}
	\begin{cases}
		\displaystyle -\Delta_p^Hu- \frac{\gamma}{H^\circ(x)^p} u^{p-1}= u^{p^*-1} \qquad &\text{in } \R^N\\
		u>0 & \text{in } \R^N\\
		u \in \mathcal{D}^{1,p}(\R^N),
	\end{cases}
\end{equation}
where $1<p<N$, $p^*:=Np/(N-p)$ is the Sobolev critical exponent, $\displaystyle 0 \leq \gamma < C_H:=\left((N-p)/p\right)^p$ is the Hardy constant and 
\begin{equation}\label{eq:anisotropic}\Delta_p^Hu:=\operatorname{div}(H^{p-1}(\nabla u)\nabla H(\nabla u)),
\end{equation}
where $\Delta_p^H$ is the so-called anisotropic p-Laplacian or Finsler p-Laplacian. 
We point out that $H$ is a Finsler type norm and $H^\circ$ is its the dual norm ($H$ satisfies assumptions $(h_H)$, see Section \ref{notations} for further details).  In particular, when $H(\xi)=|\xi|=H^\circ(\xi)$ the Finsler type p-Laplacian coincides with the classical p-Laplacian, and, hence it is singular when $1<p<2$ and degenerate when $p>2$. Then, according with standard regularity theory \cite{DB,Tolk} and the regularity results in the anisotropic framework \cite{ACF,CSR}, we say that any  solution of \eqref{eq:ProbDoublyCritical} has to be understood in the weak distributional meaning, i.e. $u \in D^{1,p}(\R^N)$ satisfies the following integral equality
\begin{equation}\label{eq:weakDoublyCritical}
	\int_{\R^N}\left(H^{p-1}(\nabla u) \langle \nabla H(\nabla u) , \nabla \varphi \rangle - \frac{\gamma}{H^\circ(x)^p} u^{p-1} \varphi\right) \, dx= \int_{\R^N} u^{p^*-1} \varphi \, dx \quad \forall \varphi \in C^\infty_c(\R^N). 
\end{equation}

The literature about critical problems is really huge. Going back to the Euclidean framework, i.e. when we consider $H(\xi)=|\xi|=H^\circ(\xi)$ in \eqref{eq:ProbDoublyCritical}, we deal with
\begin{equation}\label{eq:doublicritical}
	-\Delta_p u - \frac{\gamma}{|x|^p} u^{p-1} = u^{p^*-1} \qquad \text{in } \R^N.
\end{equation}
In the seminal paper \cite{CGS}, Caffarelli, Gidas and Spruck classified any positive solution to \eqref{eq:doublicritical} with $p=2$, $N \geq 3$, and $\gamma=0$. We point out that a first result, under stronger assumption on the decay of solutions, was obtained by Gidas, Ni and Nirenberg in \cite{GNN81}. Moreover, in this setting a complete answer in the subcritical case was done in the celebrasted work of Gidas and Spruck \cite{GS}, where the authors proved Liouville-type theorems.  

In the quasilinear framework, the situation is much more involved due to the nonlinear nature of the operator. Recently, a classification result of positive solutions to \eqref{eq:doublicritical} with $p > 2$, $\gamma=0$ and $u \in \mathcal{D}^{1,p}(\R^N):=\{u \in L^{p^*}(\R^N) \ | \ \nabla u \in L^p(\R^N)\}$ was obtained in \cite{Sciu16}. The proof of this result is based on a refined version of the well-known moving plane method of Alexandrov-Serrin \cite{A, serrinMov} and on some a priori estimates of the solutions and their gradients, proved in \cite{Vet}. To be more precise, we note that the classification result of positive solution to the Sobolev critical quasilinear equation with finite energy started in \cite{DMMS} in the case, and then was extended in \cite{Vet} for every $1<p<2$. Subsequently the full case was obtained in \cite{Sciu16}. Recently, we refer to the papers \cite{Cat, Ou, Vet2} for new partial results on the classification of positive solutions without a priori assumption on the energy of solutions. In the anisotropic setting, Ciraolo, Figalli and Roncoroni \cite{CFR}, obtained a complete classification result for positive solution to \eqref{eq:ProbDoublyCritical} with $\gamma=0$ using different techniques that do not require the use of the moving plane method, which could not be used in the anisotropic context due to the lack of invariance.

When $\gamma \neq 0$ the situation is really different. In the seminal paper of Terracini \cite{terracini}, it was proved for the first time the classification result for positive solutions to \eqref{eq:doublicritical} in the case $p=2$. The author firstly showed the existence of solutions to this problem with a minimization argument based on the concentration and compactness principle. Subsequently, she proved that any solution to this problem is radial and radially decreasing about the origin combining the moving-plane technique and the use of the Kelvin transformation, in the same spirit of \cite{CGS}. The case $p \neq 2$ and $\gamma \neq 0$ is much more involved and it is available in \cite{OSV}, where the techniques used are mainly based on a fine asymptotic analysis at infinity and refined versions of the moving plane procedure, and also on some asymptotic estimates proved in \cite{Xiang, XiangBis}.

Our aim is to prove some decay estimates for positive weak solutions to \eqref{eq:ProbDoublyCritical} in the anisotropic framework  $1<p<N$ and $\gamma\neq 0$. More precisely, our first main result is the following:

\begin{thm} \label{thm:asymptEst}
	Let $u \in \mathcal{D}^{1,p}(\R^N)$ be a weak solution of \eqref{eq:ProbDoublyCritical}
    with $1<p<N$, $0 \leq \gamma < C_H$. Then there exist positive constants $0 < R_1< 1 < R_2$ depending on $N, p, \gamma$ and  $u$, such that
    \begin{equation}\label{eq:estat0}
    	\frac{c_1}{[H^\circ(x)]^{\mu_1}}   \leq u(x) \leq \frac{C_1}{[H^\circ(x)]^{\mu_1}}  \qquad x \in \mathcal{B}_{R_1}^{H^\circ},
    \end{equation}
	and
	\begin{equation}\label{eq:estatInf}
		\frac{c_2}{[H^\circ(x)]^{\mu_2}}  \leq u(x) \leq \frac{C_2}{[H^\circ(x)]^{\mu_2}} \qquad x \in (\mathcal{B}_{R_2}^{H^\circ})^c,
	\end{equation}
where $\mu_1$, $\mu_2$ are the solutions of
\begin{equation}\label{eq:algebraicHardy}
    \mu^{p-2} [(p-1)\mu^2-(N-p)\mu] + \gamma=0,
\end{equation}
$C_1, C_2$ are positive constants depending on $N,p, \gamma, H$ and $u$, $c_1$ is a positive constant depending on $N,p, \gamma, H, R_1, \mu_1$ and $u$, $c_2$ is a positive constant depending on $N,p, \gamma, H, R_2, \mu_2$ and $u$.
\end{thm}

\begin{rem}
	In the following we shall assume that $\mu_1<\mu_2$ and
	it is easy to see that
	$$0 \leq \mu_1 < \frac{N-p}{p} < \mu_2 \leq \frac{N-p}{p-1}.$$
	furthermore $\mathcal{B}_R^{H^\circ}$ is the dual anisotropic ball also known as Frank diagram (see Section \ref{notations} for further details).
\end{rem}

In the proof we will also exploit some clever ideas from  \cite{Xiang} facing the difficulties of the anisotropic issue. A different approach is in fact needed for the study of the asymptotic behaviour of the gradient. In particular, the fact  that the moving plane plane technique cannot be applied, a crucial point is given by the following classification result:

\begin{thm}\label{introtecnicanew}
	Let $v\in C^{1,\alpha}_{loc}(\R^N\setminus \{0\})$ be a positive weak solution of the equation
	\begin{equation}\label{introzioa}
		-\Delta_p^H v - \frac{\gamma}{[H^\circ(x)]^p} v^{p-1} = 0 \  \text{in } \R^N\setminus \{0\},
	\end{equation}
	where $0 \leq \gamma < C_H$. Assume that there exist two positive constants $C$ and $c$ such that 
	\begin{equation}\label{introa_1}
		\frac{c}{[H^{\circ}(x)]^{\mu_i}}\le v(x)\le \frac{C}{[H^{\circ}(x)]^{\mu_i}} \ \forall x\in \R^N\setminus \{0\},
	\end{equation}
	where $\mu_i$ ($i=1,2$) are the roots of \eqref{eq:algebraicHardy} and suppose that there exists a positive constant $\hat C$ such that 
	\begin{equation}\label{introassunzionegradiente}
		|\nabla v(x)|\le \frac{\hat C}{[H^{\circ}(x)]^{\mu_i+1}} \ \ \forall x\in \R^N\setminus \{0\}, 
	\end{equation}
	then 
	\begin{equation}\label{introv_1}
		v(x)=\frac{\overline c}{[H^{\circ}(x)]^{\mu_i}}, \qquad \qquad \qquad 
	\end{equation}
	for some $\bar c>0$.
\end{thm}
Theorem \ref{introtecnicanew} is new and interesting in itself. The proof is very much different than the ones available in the euclidean case $H(\xi)=|\xi|$. Here we shall exploit it to deduce 
the precise asymptotic estimates for the gradient. more precisely we have the following:

\begin{thm}\label{stimagradiente}
    Let $u \in \mathcal{D}^{1,p}(\R^N)$ be a weak solution of \eqref{eq:ProbDoublyCritical}
    with $1<p<N$, and $0 \leq \gamma < C_H$. Then there exist positive constants $\tilde c$, $\tilde C$ depending on $N,p, \gamma, H$ and $u$ such that
    \begin{equation}\label{eq:estat0grad}
    	\frac{\tilde c}{[H^\circ(x)]^{\mu_1+1}}   \leq |\nabla u(x)| \leq \frac{\tilde  C}{[H^\circ(x)]^{\mu_1+1}}  \qquad x \in \mathcal{B}_{R_1}^{H^\circ},
    \end{equation}
	and
	\begin{equation}\label{eq:estatInfgrad}
		\frac{\tilde c}{[H^\circ(x)]^{\mu_2+1}}  \leq |\nabla u(x)| \leq \frac{\tilde C}{[H^\circ(x)]^{\mu_2+1}} \qquad x \in (\mathcal{B}_{R_2}^{H^\circ})^c,
	\end{equation}
where $\mu_1, \mu_2$ are roots of \eqref{eq:algebraicHardy} as in Theorem \ref{thm:asymptEst}, and $0 < R_1< 1 < R_2$ are constants depending on $N, p, \gamma$ and  $u$.
\end{thm}

The paper is structured as follows:

\begin{itemize}
	\item In Section \ref{notations} we recall some notions about Finsler type anisotropic  geometry, and we prove some technical lemmas that will be crucial in the proof of the main results.

	\item In Section \ref{preliminaryEst} we prove some preliminary estimates, elliptic estimates and weak comparison principles in bounded and exterior domains that will be essential in the proof of Theorem \ref{thm:asymptEst}.

	\item In Section \ref{AsympEst} we give the proof of decay estimates of solutions  to \eqref{eq:ProbDoublyCritical} near the origin and at infinity, i.e. we prove Theorem \ref{thm:asymptEst}. The, using this result we also prove decay estimates for the gradient of positive weak solutions to \eqref{eq:ProbDoublyCritical} near the origin and at infinity, i.e. we prove Theorem \ref{stimagradiente}. 

     \item Although the existence of solutions can be easy deduced in the radial-anisotropic setting, in the Appendix \ref{appendix}  we show that problem \eqref{eq:ProbDoublyCritical} admits at least a positive solution $u \in D^{1,p}(\R^N)$ that minimizes the Hardy-Sobolev anisotropic inequality. This result follows using classical arguments (see also \cite{terracini}) that we decide to add for the readers' convenience.
    
\end{itemize}

\section{Preliminaries}\label{notations}
\noindent {\bf Notation.} Generic fixed and numerical constants will
be denoted by $C$ (with subscript in some case) and they will be
allowed to vary within a single line or formula. By $|A|$ we will
denote the Lebesgue measure of a measurable set $A$.\\

The aim of this section is to recall some properties and geometrical tools about the anisotropic elliptic operator defined above. For $a, b \in \R^N$ we denote by $a \otimes b$ the matrix whose entries are $(a \otimes b)_{ij}=a_ib_j$. We remark that for any $v,w \in \R^N$ it holds that:
$$\langle a \otimes b \ v, w \rangle = \langle b, v \rangle \langle a , w \rangle.$$

Now, we recall the definition of  anisotropic norm.

\begin{itemize}
	\item[$(h_H)$]
Let $H \in C^2(\R^N \setminus \{0\})$. In all the paper we assume that $H$ is a anisotropic norm if it satisfies the following set of assumptions:

\begin{enumerate}
    \item[(i)] $H(\xi)>0 \quad \forall \xi \in \R^N \setminus \{0\}$;

    \item[(ii)] $H(s \xi) = |s| H(\xi) \quad \forall \xi \in \R^N \setminus \{0\}, \, \forall s \in \R$;

    \item[(iii)] $H$ is \emph{uniformly elliptic}, that means the set $\mathcal{B}_1^H:=\{\xi \in \R^N  :  H(\xi) < 1\}$ is \emph{uniformly convex}
    \begin{equation}
        \exists \Lambda > 0: \quad \langle D^2H(\xi)v, v \rangle \geq \Lambda |v|^2 \quad \forall \xi \in \partial \mathcal{B}_1^H, \; \forall v \in \nabla H(\xi)^\bot.
    \end{equation}
\end{enumerate}
\end{itemize}

A set is said uniformly convex if the principal curvatures of its boundary are all strictly positive.
Moreover, assumption (iii) is equivalent to assume that $D^2 (H^2)$ is definite positive.

	The {\em dual norm} $H^\circ : \R^N \rightarrow [0,+\infty)$ is defined as:
	$$H^\circ (x) = \sup_{H(\xi) \leq 1} \langle \xi , x \rangle.$$

It is possible to show that $H^\circ$ is also a Finsler norm and it has the same regularity properties of $H$. Moreover, it holds $(H^\circ)^\circ = H$. For $R > 0$ and $\bar x \in \R^N$ we define:
$$\mathcal{B}_R^H(\bar x) = \{ x \in \R^N \ : \ H(x-\bar x) < R\}$$
and 
$$\mathcal{B}_R^{H^\circ} (\bar x) = \{ x \in \R^N \ : \ H^\circ(x-\bar x) < R\}.$$

For simplicity of exposition, when $\bar x = 0$, we set: $\mathcal{B}_R^H=\mathcal{B}_R^H(0)$, $\mathcal{B}_R^{H^\circ}=\mathcal{B}_R^{H^\circ}(0)$. In literature  $\mathcal{B}_R^H$ and  $\mathcal{B}_R^{H^\circ}$ are also called ``Wulff shape" and ``Frank diagram" respectively. We remark that there holds the following identities:
\begin{equation}\label{eq:PropFinsler}
H(\nabla H ^\circ (x)) = 1 = H^\circ (\nabla H (x));
\end{equation}
and 
\begin{equation}\label{eq:PropFinsler2}
    H(x)\nabla H^\circ(\nabla H(x))=x=H^\circ(x)\nabla H(\nabla H^\circ (x)).
\end{equation}
We refer the reader to \cite{BellPao, CianSal} for further details. 
Observe also that $H$ is a norm equivalent to the euclidean one, i.e.  there exist $\alpha_1, \alpha_2 > 0$ such that:
\begin{equation}\label{eq:equivalentNorm}
	\alpha_1 |\xi| \leq H(\xi) \leq \alpha_2 |\xi| \qquad \forall \xi \in \R^N.
\end{equation}
Moreover, recalling that $H$ is $1$-homogeneous, by the Euler's Theorem it follows
\begin{equation}\label{eq:EulerThm}
	\langle \nabla H(\xi), \xi \rangle = H(\xi) \qquad \forall \xi \in \R^N \setminus \{0\}.
\end{equation}
Since $H$ is $1$-homogeneous, we have that $\nabla H$ is $0$-homogeneous and it satisfies
$$\nabla H (\xi) = \nabla H \left(|\xi|\frac{\xi}{|\xi|}\right)= \nabla H \left(\frac{\xi}{|\xi|}\right) \qquad \forall \xi \in \R^N\setminus \{0\}.$$
Hence, by the previous equality, we infer that there exists $M>0$ such that
\begin{equation}\label{eq:BddH}
	| \nabla H(\xi)| \leq M \qquad \forall \xi \in \R^N\setminus \{0\}.
\end{equation}
For the same reasons there exists a constant $\overline M>0$ such that:
\begin{equation}\label{eq:PropFinsler11}
|D^2H(\xi)|\le \frac{\overline M}{|\xi|}\qquad \forall\,\xi\in\R^N\setminus\{0\}, 
\end{equation}
where $|\cdot|$ denotes the usual Euclidean  norm of a matrix, and
\begin{equation}\label{eq:PropFinsler22}
    D^2H(\xi)\xi=0 \qquad \forall\,\xi\in\R^N\setminus\{0\}.
\end{equation}

 We start with some elliptic estimates that can be proved in the same spirit of the Euclidean framework.

\begin{prop} \label{prop:Lindqvist} 
For any $p>1$ and  $\eta, \eta' \in \R^N$ such that $|\eta|+|\eta'|>0$, it holds 
\begin{equation}\label{zaia}
    |H^{p-1}(\eta)\nabla H(\eta)-H^{p-1}(\eta ')\nabla H(\eta ')|\le \tilde C_p(|\eta|+|\eta'|)^{p-2}|\eta-\eta'|.
\end{equation}

Moreover, any $p \geq 2$ it holds the following inequality
\begin{equation} \label{introeq:ineqStandard}
	H^p(\eta) \geq H^p(\eta') + p H^{p-1}(\eta') \langle \nabla H(\eta'), \eta -\eta' \rangle + \hat C(p) H^p(\eta - \eta'), 
\end{equation}
for any $\eta, \eta' \in \R^N$. Furthermore, if $1<p<2$ we have that
\begin{equation} \label{intro:eq:ineqStandardBis}
	H^p(\eta) \geq H^p(\eta') + p H^{p-1}(\eta') \langle \nabla H (\eta'), \eta-\eta' \rangle + C_p [H(\eta) + H(\eta')]^{p-2}H^2(\eta - \eta'),
\end{equation}
for any $\eta, \eta' \in \R^N$ such that $|\eta|+|\eta'|>0$.
\end{prop}

\begin{proof}
     
We start the proof showing \eqref{zaia}. First of all we note that \eqref{zaia} is symmetric in $\eta$, $\eta'$. Hence, without loss of generality, we can assume that $|\eta'| \geq |\eta| > 0$. We note that for $j=1, \dots ,N$:
\begin{equation}\label{eq:Lagrangethm}
	\begin{split}
	&H^{p-1}(\eta) \frac{\partial H}{\partial \eta_j}(\eta)-H^{p-1}(\eta') \frac{\partial H}{\partial \eta_j}(\eta') = \\
	=& \int_0^1 \sum_{i=1}^N \left[(p-1)H^{p-2}(\eta'+t(\eta-\eta')) \left(\frac{\partial H}{\partial \eta_i}  \cdot \frac{\partial H}{\partial \eta_j}\right) (\eta'+t(\eta-\eta'))   \right.\\
	& \qquad \qquad  \left. + H^{p-1}(\eta'+t(\eta-\eta')) \frac{\partial^2 H}{\partial \eta_i \partial \eta_j}(\eta'+t(\eta-\eta')) \right] \cdot (\eta_i-\eta_i')  \, dt.
	\end{split}
\end{equation}

By \eqref{eq:Lagrangethm}, using \eqref{eq:equivalentNorm}, \eqref{eq:BddH} and \eqref{eq:PropFinsler11} we have
\begin{equation} \label{eq:LagrangeThmMonotonicityBis}
	\begin{split}
		&|H^{p-1}(\eta)\nabla H(\eta)-H^{p-1}(\eta')\nabla H(\eta')| \leq \\
		& \leq \int_0^1 |(p-1)H^{p-2}(\eta'+t(\eta-\eta')) \nabla H(\eta'+t(\eta-\eta')) \otimes \nabla H(\eta'+t(\eta-\eta')) \\
		& \qquad \ \ \ + H^{p-1}(\eta'+t(\eta-\eta')))D^2H(\eta'+t(\eta-\eta'))| \cdot |\eta -\eta'| \, dt\\
		& \leq ((p-1)\alpha_2^{p-2}M^2+\alpha_2^{p-1}\overline M) |\eta - \eta'| \int_0^1 (|\eta'+t(\eta-\eta')|)^{p-2} \, dt,
	\end{split}
\end{equation}
where $|\cdot|$ denotes the standard matrix Euclidean norm.

Now, we observe that 
\begin{equation} \label{eq:triangularinequality}
	 |\eta'+t(\eta-\eta')| \leq  |\eta|+ |\eta'|,
\end{equation}
and, since $|\eta'| \geq |\eta|$, we have either
\begin{equation} \label{eq:reversetriangularinequality}
	|\eta - \eta'| \leq \frac{|\eta'|}{2} \ \ \Longrightarrow \ \ |\eta'+t(\eta-\eta')| \geq |\eta'|-|\eta-\eta'| \geq  \frac{|\eta'|}{2} \geq \frac{|\eta'|+|\eta|}{4}
\end{equation}
or, putting $\displaystyle t_0:= \frac{|\eta'|}{|\eta-\eta'|} \in (0,2)$,
\begin{equation} \label{eq:reversetriangularinequalitybis}
	\begin{split}
     	|\eta - \eta'| > \frac{|\eta'|}{2} \ \ \Longrightarrow \ \   |\eta'+t(\eta-\eta')| & \geq \ \left||\eta'|-t|\eta-\eta'|\right| = |t_0-t| \cdot | \eta -\eta'| \\
     	&  \geq |t_0-t| \cdot \frac{|\eta'|}{2} = |t_0-t|\frac{|\eta'|}{2} \\
     	& \geq |t_0-t| \frac{|\eta| + |\eta'|}{4}. 
	\end{split}
\end{equation}

If $p>2$, using \eqref{eq:triangularinequality} in \eqref{eq:LagrangeThmMonotonicityBis} we have
\begin{equation} \label{eq:LagrangeThmMonotonicity4}
	|H^{p-1}(\eta)\nabla H(\eta)-H^{p-1}(\eta')\nabla H(\eta')| \leq
		 \tilde C_p (|\eta '|+|\eta|)^{p-2}|\eta - \eta'| 
	\end{equation}
where $\tilde C_p=(p-1)\alpha_2^{p-2}M^2+\alpha_2^{p-1}\overline M$. Hence \eqref{zaia} holds.

If $p \leq 2$ and \eqref{eq:reversetriangularinequality} holds, by \eqref{eq:LagrangeThmMonotonicityBis} we obtain 
\begin{equation} \label{eq:LagrangeThmMonotonicityTris}
	\begin{split}
		&|H^{p-1}(\eta)\nabla H(\eta)-H^{p-1}(\eta')\nabla H(\eta')| \leq \\
		& \leq ((p-1)\alpha_2^{p-2}M^2+\alpha_2^{p-1}\overline M) |\eta - \eta'| \int_0^1 (  |\eta'+t(\eta-\eta')|)^{p-2} \, dt.\\
		& \leq ((p-1)\alpha_2^{p-2}M^2+\alpha_2^{p-1}\overline M) |\eta - \eta'| \int_0^1 \left(\frac{|\eta'|+|\eta|}{4}\right)^{p-2} \, dt\\
		& \leq \tilde C_p (|\eta'|+|\eta|)^{p-2}|\eta-\eta'|,
	\end{split}
\end{equation}
where $\tilde C_p=((p-1)\alpha_2^{p-2}M^2+\alpha_2^{p-1}\overline M)/4^{p-2}$. Hence \eqref{zaia} holds \\

If $p \leq 2$ and \eqref{eq:reversetriangularinequalitybis} holds, by \eqref{eq:LagrangeThmMonotonicityBis} we obtain 
\begin{equation} \label{eq:LagrangeThmMonotonicityPoker}
	\begin{split}
		&|H^{p-1}(\eta)\nabla H(\eta)-H^{p-1}(\eta')\nabla H(\eta')| \leq \\
		& \leq  ((p-1)\alpha_2^{p-2}M^2+\alpha_2^{p-1}\overline M)|\eta - \eta'| \int_0^1 ( |\eta'+t(\eta-\eta')|)^{p-2} \, dt.\\
		& \leq ((p-1)\alpha_2^{p-2}M^2+\alpha_2^{p-1}\overline M) |\eta - \eta'| \int_0^1 |t_0-t|^{p-2}\left(\frac{|\eta'|+|\eta|}{4}\right)^{p-2} \, dt\\
		& = \frac{(p-1)\alpha_2^{p-2}M^2+\alpha_2^{p-1}\overline M}{4^{p-2}} ( |\eta'|+|\eta|)^{p-2}|\eta-\eta'| \int_0^1 |t_0-t|^{p-2} \, dt\\
		&\leq \frac{(p-1)\alpha_2^{p-2}M^2+\alpha_2^{p-1}\overline M}{4^{p-2}} (|\eta'|+|\eta|)^{p-2}|\eta-\eta'|  \ 2 \int_0^1  z^{p-2} \, dz\\
		& = \tilde C_p (|\eta'|+|\eta|)^{p-2}|\eta-\eta'|,  
	\end{split}
\end{equation}
 where $\tilde C_p=(2(p-1)\alpha_2^{p-2}M^2+2\alpha_2^{p-1}\overline M)/4^{p-2}$.\\

Collecting the estimates above, we deduce that inequality \eqref{zaia} holds for every $p>1$ and for $\tilde C_p=((p-1)\alpha_2^{p-2}M^2+\alpha_2^{p-1}\overline M) \cdot \max\{1, 4^{2-p}, 2 \cdot 4^{2-p}\}$.

Now we will show \eqref{introeq:ineqStandard} and \eqref{intro:eq:ineqStandardBis}. For $\eta, \eta'\in \R^N$, we define $$f(t)=H^p(\eta'+t(\eta-\eta'))$$ and Taylor's formula yields 
\begin{equation}
  H^p(\eta) = H^p(\eta') + p H^{p-1}(\eta') \langle   \nabla H(\eta'), \eta -\eta' \rangle + \int_0^1 (1-t)f''(t) \,dt ,
\end{equation}
provided $|\eta'+t(\eta-\eta')|\neq 0$, for $0\le t\le 1$. But the case when $\eta'+t(\eta-\eta')=0$ can be easily verified. By \cite[Theorem $1.5$]{CFV} we obtain that 
\begin{equation}
\begin{split}
    f''(t)&=\langle D^2(H^p(\eta'+t(\eta-\eta')))(\eta-\eta'),\eta-\eta'\rangle \\ 
     &\ge C(p)|\eta'+t(\eta-\eta')|^{p-2}|\eta-\eta'|^2,
    \end{split}
\end{equation}
where $C(p)$ is a constant depending on $p$.
We remark that  
\begin{equation}\label{botto}
    \int_0^1(1-t)f''(t) \,dt \ge \frac{3}{4}\int_0^{\frac{1}{4}}f''(t) \,dt.
\end{equation}

If $1<p<2$, by \eqref{eq:triangularinequality} we have $$(H(\eta)+H(\eta '))^{p-2}\le \alpha_2^{p-2}|\eta'+t(\eta-\eta')|^{p-2}$$ and using \eqref{botto} we arrive at \eqref{intro:eq:ineqStandardBis}.

If $p\ge 2$, using a similar argument as in the proof of the inequality \eqref{zaia}, we obtain 
\begin{equation}
   \int_0^{\frac{1}{4}}f''(t)\ge \overline C(p)(|\eta|+|\eta'|)|^{p-2}|\eta-\eta'|^2,
\end{equation}
with $\overline C(p)$ constant depending on $p$. Since $|\eta-\eta'|\le |\eta|+|\eta '|$ and using \eqref{eq:equivalentNorm}, we get \eqref{introeq:ineqStandard}.
\end{proof}

We state now the Hardy inequality for the anisotropic operator $\Delta_p^Hu$, defined in  \eqref{eq:anisotropic}. We refer to \cite[Proposition 7.5]{VanHardy}.
\begin{thm}[Hardy inequality]\label{thm:Hardy}
		For any $H$ satisfying the assumption $(h_H)$ and any $u\in \mathcal{D}^{1,p}(\R^N)$ and $1<p<N$,
		\begin{equation}\label{eq:Hardy}
			C_H\int_{\mathbb R^N}\frac{|u|^p}{{H^\circ(x)^p}}\, dx\leq \int_{\mathbb R^N}H^p(\nabla u)\, dx,
		\end{equation}
		where $C_H=(({N-p})/{p})^p$ is optimal. 
\end{thm}

Now we prove a technical lemma that will be very important in the proof of the asymptotic estimates.

\begin{lem}\label{lem:a^p>b^p}
	Let $p>1$ and $a,b \geq 0$. Then, for all $\delta>0$ there exist $\mathcal{C}_\delta>0$ such that
	\begin{equation}\label{eq:a^p>b^p}
		a^p \geq \frac{1}{1+ 2^{p+1}\delta} (a+b)^p - \mathcal{C}_\delta b^p.
	\end{equation}
\end{lem}

\begin{proof}
	Let us consider $p>1$ as  follows:
	$$p=\lfloor p \rfloor + \{p\},$$
	where $\lfloor \cdot \rfloor$ is the floor function and $\{\cdot\}$ is the mantissa function. Without loss of generality we assume that $\{p\} \neq 0$ and, moreover, we set $m:= \lfloor p \rfloor$. Hence, we have
	\begin{equation}\label{eq:Newton1}
		(a+b)^p=(a+b)^{\{p\}}(a+b)^m= (a+b)^{\{p\}} \sum_{k=0}^m {m \choose k} a^{m-k} b^k = (\star)
	\end{equation}
	Noticing that $0 \leq \{p\} < 1$ it follows that
	$$(a+b)^{\{p\}} \leq a^{\{p\}} + b^{\{p\}}.$$
	Using this inequality in \eqref{eq:Newton1}, we deduce
	\begin{equation}\label{eq:Newton2}
		\begin{split}
		(\star) \leq& a^{\{p\}} \sum_{k=0}^m {m \choose k} a^{m-k} b^k + b^{\{p\}} \sum_{k=0}^m {m \choose k} a^{m-k} b^k\\
		=& a^p+b^p+\sum_{k=1}^m {m \choose k} a^{p-k} b^k +  \sum_{k=0}^{m-1} {m \choose k} a^{m-k} b^{k+\{p\}}=(\ast),
		\end{split}
	\end{equation}

    where we used the fact that $p=m+\{p\}$.
    Now, we can apply the weighted Young's inequality to each member of the first sum with conjugate exponents $\left(p/(p-k),p/k\right)$ and to each member of the second sum with conjugate exponents $\left(p/(m-k),p/(k+\{p\})\right)$ 
    as follows
    \begin{equation}
    	\begin{split}
    		a^{p-k}b^k &\leq \frac{p-k}{p} \delta a^p + \frac{k}{p} \mathcal{C}_\delta b^p \leq \delta a^p + \mathcal{C}_\delta b^p,\\
    		a^{m-k} b^{k+\{p\}} & \leq \delta a^p + \mathcal{C}_\delta b^p.
    	\end{split}
    \end{equation}
    Hence, using this estimate we deduce  
   	\begin{equation}\label{eq:Newton3}
   		\begin{split}
   			(\ast) \leq& a^p+b^p+ \left(\delta a^p + \mathcal{C}_\delta b^p\right)  \sum_{k=1}^m {m \choose k} +  \left(\delta a^p + \mathcal{C}_\delta b^p\right) \sum_{k=0}^{m-1}  {m \choose k} \\
   			\leq& \left(1+2^{p+1} \delta \right)a^p + \mathcal{C}_\delta b^p,
   		\end{split}
   	\end{equation}
   where we renamed $\mathcal{C}_\delta:= \left(1+2^{p+1} \mathcal{C}_\delta \right)$.
   Collecting \eqref{eq:Newton1}, \eqref{eq:Newton2} and \eqref{eq:Newton3}, we deduce that 
   \begin{equation}
	a^p \geq \frac{1}{1+ 2^{p+1}\delta} (a+b)^p - \mathcal{C}_\delta b^p,
   \end{equation}
with $\mathcal{C}_\delta:=\mathcal{C}_\delta/(1+2^{p+1}\delta)$, and hence the thesis \eqref{eq:a^p>b^p}.
\end{proof}

Finally, we recall a lemma (see Lemma 4.19 in \cite{HanLin}) that will be very useful in the proofs of our results.

\begin{lem}[\cite{HanLin}]\label{lem:nondec}
Let $\mathcal{L}$ and $g$ be two nondecreasing functions on the interval $(0, \bar R]$, for some $\bar R>0$. Suppose that it holds
$$\mathcal{L}(\tau R) \leq \sigma \mathcal{L}(R) + g(R) \quad \text{for all } R \leq \bar R,$$
for some $0 < \sigma, \tau < 1$. Then, for any $\mu \in (0,1)$ and $R \leq \bar R$ we have
$$\mathcal{L}(R) \leq \frac{1}{\sigma} \left(\frac{R}{\bar R}\right)^\alpha \mathcal{L}(\bar R) + \frac{1}{1-\sigma}g(\bar R^\mu R^{1-\mu})$$
where $\alpha = \alpha (\sigma, \tau, \mu) =(1-\mu) \log \sigma / \log \tau$.
\end{lem}

\section{Preliminary asymptotic estimates and comparison principles} \label{preliminaryEst}

The aim of this section is to prove some preliminary estimates that will be crucial in the proofs of the main results.

\begin{lem}\label{lem:aux1}
	 There exists a positive constant $\tau$ depending only on $N, p$ and $\gamma$ such that for any $\bar{R}> 0$ and for any solution $u$  to problem \eqref{eq:ProbDoublyCritical} satisfying 
	\begin{equation}\label{eq:assumptionsmall}
		\|u\|_{L^{p^*}(\mathcal{B}^{H^\circ}_{\bar R})} + \|u\|_{L^{p^*}(\R^N \setminus \mathcal{B}^{H^\circ}_{1 / \bar R})} \leq \tau , 
	\end{equation} 
	 there exists a positive constant $\mathcal{C}$ depending only on $N, p, \gamma$ and $\overline R$ such that 
	\begin{equation} \label{eq:EstAt0}
		\|u\|_{L^{p^*}(\mathcal{B}^{H^\circ}_R)} \leq \mathcal{C} R^{\sigma_1}  \quad \text{for } R \leq \bar R, 
	\end{equation} 
	and that
	\begin{equation}  \label{eq:EstAtInfty}
       	\|u\|_{L^{p^*}((\mathcal{B}^{H^\circ}_R)^c)} \leq \frac{\mathcal{C}}{R^{\sigma_2}} \quad \text{for } R \geq \frac{1}{\bar R}, 
	\end{equation} 
	where $\sigma_1, \sigma_2$ are two positive constants depending on $N, p $ and $\gamma$.
\end{lem}

\begin{proof}
	We start proving \eqref{eq:EstAt0}. To this aim let us consider $ R>0$ and a cut-off function $\eta \in C^\infty_c(\R^N)$ such that
    \begin{equation}\label{eq:cutoff}
    	\left\{\begin{array}{llll}
		0 \leq \eta \leq 1 \quad &\text{in } \R^N\\
		\eta \equiv 0 \quad &\text{in } (\mathcal{B}^{H^\circ}_R)^c\\
		\eta \equiv 1 \quad &\text{in } \mathcal{B}^{H^\circ}_{R / 2}\\
		|\nabla \eta| \leq \frac{4}{R} &\text{in } \mathcal{B}^{H^\circ}_R \setminus \mathcal{B}^{H^\circ}_{R / 2}.
	\end{array}\right.
    \end{equation}
	By density argument it is possible to put $\varphi=\eta^p u$ as test function in \eqref{eq:weakDoublyCritical}, so that we obtain
	\begin{equation}\label{eq:weakForTest1}
		\begin{split}
		&\int_{\R^N} p \eta^{p-1} u H^{p-1}(\nabla u) \langle \nabla H(\nabla u) ,  \nabla \eta  \rangle  \, dx + \int_{\R^N}\eta^p H^{p-1}(\nabla u) \langle \nabla H(\nabla u) ,   \nabla u \rangle  \, dx\\
		&= \int_{\R^N} \frac{\gamma}{H^\circ(x)^p} u^{p-1} \eta^p u \, dx + \int_{\R^N} 	u^{p^*-1} \eta^p u \, dx
		\end{split}.
	\end{equation}

	First of all, using Euler's Theorem \eqref{eq:EulerThm}, the 0-homogeneity of $\nabla H$ \eqref{eq:BddH} and Schwarz's inequality, equation \eqref{eq:weakForTest1}  becomes
	\begin{equation}\label{eq:weakForTest2}
		\begin{split}
			&\int_{\R^N} H^p (\nabla u) \eta^p \, dx = \int_{\R^N}H^{p-1}(\nabla u) \langle \nabla H(\nabla u) ,   \nabla u \rangle \eta^p \, dx\\
			& \leq p \ M \int_{\R^N} H^{p-1}(\nabla u)  |\nabla \eta| \eta^{p-1} u \, dx + \int_{\R^N} \frac{\gamma}{H^\circ(x)^p} u^p \eta^p  \, dx +\int_{\R^N} 	u^{p^*} \eta^p  \, dx
		\end{split}
	\end{equation}

    Recalling that $H$ is $1$-homogeneous function, using the weighted Young's inequality $ab \leq  \varepsilon a^{\frac{p}{p-1}} + \mathcal{C}_\varepsilon b^p$ on the first term of the right hand side of \eqref{eq:weakForTest2}, for any $0<\varepsilon<1$ we have
 \begin{equation}\label{eq:weakForTest3}
	\begin{split}
		\int_{\R^N} H^p (\eta \nabla u) \, dx  & \le\varepsilon   \int_{\R^N} H^p(\eta \nabla u)  \, dx + C(p,M,\varepsilon) \int_{\R^N}  |\nabla \eta|^p u^p \, dx \\
		&\ \ + \int_{\R^N} \frac{\gamma}{H^\circ(x)^p} u^{p} \eta^p  \, dx + \int_{\R^N} 	u^{p^*} \eta^p  \, dx,\\
	\end{split}
\end{equation} 
where $ C(p,M,\varepsilon):=(pM)^p \mathcal{C}_\varepsilon$. Now, noticing that $\nabla(\eta u) = u \nabla \eta + \eta \nabla u$, by the triangular inequality, we deduce that for every $p>1$ it holds
\begin{equation}\label{eq:FinTrIn}
	 H^p(\nabla(\eta u)) =H^p(u \nabla \eta + \eta \nabla u)\leq \left[H(u \nabla \eta) +H(\eta \nabla u) \right]^p.
\end{equation}
Thanks to \eqref{eq:FinTrIn} and applying Lemma \ref{lem:a^p>b^p} with $a=H(\eta\nabla u)$ and $b=H(u\nabla \eta)$, we deduce that
\begin{equation}\label{eq:lemmaTypeYoung}
	\int_{\R^N} H^p(\eta \nabla u) \,dx \geq \frac{1}{1+2^{p+1}\delta} \int_{\R^N} H^p(\nabla (\eta u)) \,dx - \mathcal{C}_\delta \int_{\R^N} H^p(u \nabla \eta) \, dx
\end{equation}

Using \eqref{eq:lemmaTypeYoung} in \eqref{eq:weakForTest3} we obtain
\begin{equation}\label{eq:weakForTest4}
	\begin{split}
		\frac{1-\varepsilon}{1+2^{p+1}\delta}\int_{\R^N} H^p (\nabla (\eta u))  \, dx \leq& \mathcal{C}_\delta \int_{\R^N} H^p(u \nabla \eta) \, dx + C(p,M,\varepsilon)  \int_{\R^N}  |\nabla \eta|^p u^p \, dx \\
		&+ \gamma \int_{\R^N} \frac{(u \eta)^p}{H^\circ(x)^p}   \, dx  + \int_{\R^N} 	u^{p^*} \eta^p  \, dx\\
	\end{split}
\end{equation}

Now, applying the anisotropic Hardy inequality (see Theorem \ref{thm:Hardy} or \cite{VanHardy}) and \eqref{eq:equivalentNorm} we have
\begin{equation}\label{eq:weakForTest5}
	\begin{split}
		\frac{1-\varepsilon}{1+2^{p+1}\delta}\int_{\R^N} H^p (\nabla (\eta u))  \, dx  \leq&  \left(\alpha_2^p \mathcal{C}_\delta +  C(p,M,\varepsilon) \right) \int_{\R^N}  |\nabla \eta|^p u^p \, dx \\
		&+ \frac{\gamma}{C_H}\int_{\R^N} H^p (\nabla (\eta u))  \, dx + \int_{\R^N} 	u^{p^*} \eta^p  \, dx.
	\end{split}
\end{equation} 
Let us fix $\varepsilon, \delta >0$ sufficiently small such that $\mathcal{C}_1:=({1-\varepsilon})/({1+2^{p+1}\delta}) - {\gamma}/{C_H} > 0$, so that
\begin{equation}\label{eq:weakForTest6}
		\mathcal{C}_1 \int_{\R^N} H^p (\nabla (\eta u))  \, dx  \leq  \mathcal{C}_2 \int_{\R^N}  |\nabla \eta|^p u^p \, dx  + \int_{\R^N} 	u^{p^*} \eta^p  \, dx,
\end{equation}
where $\mathcal{C}_2:=  C(p,M,\varepsilon) + \alpha_2^p \mathcal{C}_\delta$.
By \eqref{eq:equivalentNorm} we have
\begin{equation}\label{eq:weakForTest7}
	\alpha_1^p\mathcal{C}_1 \int_{\R^N} |\nabla (\eta u)|^p  \, dx  \leq  \mathcal{C}_2 \int_{\R^N}  |\nabla \eta|^p u^p \, dx  + \int_{\R^N} 	u^{p^*} \eta^p  \, dx.
\end{equation}
Now, using the Sobolev inequality in the left hand side of \eqref{eq:weakForTest7}, the H\"older inequality and \eqref{eq:cutoff} in the right hand side, we obtain
\begin{equation}\label{eq:weakForTest8}
	\begin{split}
	\alpha_1^p\mathcal{C}_1 \mathcal{C}_S^p \left(\int_{\R^N} |\eta u|^{p^*} \,dx \right)^\frac{p}{p^*} &\leq \alpha_1^p\mathcal{C}_1 \int_{\R^N} |\nabla (\eta u)|^p  \, dx  \\
	&\leq  \mathcal{C}_2 \int_{\R^N}  |\nabla \eta|^p u^p \, dx  + \int_{\R^N} 	u^{p^*-p} (\eta u)^p  \, dx\\
	& \leq \mathcal{C}_2 \left(\int_{\mathcal{B}^{H^\circ}_R \setminus \mathcal{B}^{H^\circ}_{R / 2}}  |\nabla \eta|^N  \, dx\right)^{\frac{p}{N}} \left(\int_{\mathcal{B}^{H^\circ}_R \setminus \mathcal{B}^{H^\circ}_{R / 2}}  	u^{p^*}  \, dx\right)^{\frac{p}{p^*}}  \\
	& \ \ + \left(\int_{\mathcal{B}^{H^\circ}_R}  u^{p^*}  \, dx\right)^{\frac{p}{N}} \left(\int_{\R^N}  (\eta u)^{p^*}  \, dx\right)^{\frac{p}{p^*}}\\
	& \leq \mathcal{C}_2 C(p,N) \left(\int_{\mathcal{B}^{H^\circ}_R \setminus \mathcal{B}^{H^\circ}_{R / 2}}  	u^{p^*}  \, dx\right)^{\frac{p}{p^*}} \\
	& \ \ + \|u\|_{p^*}^{p^*-p} \left(\int_{\R^N}  (\eta u)^{p^*}  \, dx\right)^{\frac{p}{p^*}},
\end{split}
\end{equation}
hence we deduce
\begin{equation}\label{eq:weakForTest9}
	\begin{split}
		\left(\int_{\R^N} |\eta u|^{p^*} \,dx \right)^\frac{p}{p^*} \leq \frac{\mathcal{C}_2 C(p,N)}{\alpha_1^p\mathcal{C}_1 \mathcal{C}_S^p} \left(\int_{\mathcal{B}^{H^\circ}_R \setminus \mathcal{B}^{H^\circ}_{R / 2}}  	u^{p^*}  \, dx\right)^{\frac{p}{p^*}} + \frac{\|u\|_{L^{p^*}(\mathcal{B}^{H^\circ}_R)}^{p^*-p}}{\alpha_1^p\mathcal{C}_1 \mathcal{C}_S^p} \left(\int_{\R^N}  (\eta u)^{p^*}  \, dx\right)^{\frac{p}{p^*}},
	\end{split}
\end{equation}
where $C(p,N)$ is a positive constant depending on $p$ and $N$.

Setting
$$\tau = \left(\frac{\alpha_1^p\mathcal{C}_1 \mathcal{C}_S^p}{2}\right)^{\frac{1}{p^*-p}},$$
and choosing $\overline{R}>0$ sufficiently small such that \eqref{eq:assumptionsmall}
holds, then $\|u\|_{L^{p^*}(\mathcal{B}^{H^\circ}_R)}^{p^*-p}/(\alpha_1^p\mathcal{C}_1 \mathcal{C}_S^p) \leq 1/2$  for all $0<R \leq \overline R$. Hence we obtain that
$$\int_{\mathcal{B}^{H^\circ}_{R/2}} u^{p^*} \,dx  \leq \int_{\R^N} |\eta u|^{p^*} \,dx \leq \bar C \int_{\mathcal{B}^{H^\circ}_R \setminus \mathcal{B}^{H^\circ}_{R / 2}}  	u^{p^*}  \, dx \qquad \forall \  0 < R \leq \bar R,$$
where $\bar C := \left((2\mathcal{C}_2 C(p,N))/({\alpha_1^p\mathcal{C}_1 \mathcal{C}_S^p})\right)^{\frac{p^*}{p}}$ and it depends only on $N$, $p$ and $\gamma$. Denoting with $\displaystyle \mathcal{L}(R):=\int_{\mathcal{B}^{H^\circ}_R}u^{p^*} \, dx$ for $0 < R \leq \bar R$, we get that
$$\mathcal{L}(R/2) \leq \vartheta \mathcal{L}(R) \qquad \forall \ 0 < R \leq \bar R,$$
where $\vartheta = \bar C/(\bar C + 1)\in (0,1)$, depends only on $N$, $p$ and $\gamma$. Now, by Lemma \ref{lem:nondec} it follows that
$$\mathcal{L}(R) \leq \frac{1}{\vartheta} \mathcal{L}(\bar R) \left(\frac{R}{\bar R}\right)^{\sigma_1 '} \qquad \forall \ 0 < R \leq \bar R,$$
where $\sigma_1'=\frac{1}{2}\log(1 / \vartheta)/ \log 2$ depends only on $\vartheta$, Now \eqref{eq:EstAt0} follows by setting $\sigma_1 = \sigma_1' / p^*$ and $\mathcal{C}=(\vartheta^{-1}\overline R^{-\sigma_1'}\mathcal{L}(\bar R))^{1/p^*}$. In a similar way, we can deduce \eqref{eq:EstAtInfty}.
\end{proof}

Now, we denote by $\mathcal{A}_R=\mathcal{B}^{H^\circ}_{8R} \setminus \mathcal{B}^{H^\circ}_{R/8}$ and $\mathcal{D}_R=\mathcal{B}^{H^\circ}_{4R} \setminus \mathcal{B}^{H^\circ}_{R/4}$ for $R>0$. 
\begin{lem} \label{aux2}
	Let $t \in (p^*, N / \mu_1)$. There exists a positive constant $\sigma = \sigma(N,p,\gamma,t)$ such that for any solution $u$  to problem \eqref{eq:ProbDoublyCritical} and for any $\bar R> 0$ satisfying the following inequality 
	\begin{equation}\label{eq:assumptionsmallaux2}
		\|u\|_{L^{p^*}(B_{\bar R})} + \|u\|_{L^{p^*}(\R^N \setminus B_{1 / \bar R})} \leq \sigma,
	\end{equation} 
    then 
	\begin{equation} \label{eq:EstAt0aux2}
		\left(\fint_{\mathcal{D}_R} u^t \, dx \right)^{\frac{1}{t}} \leq \mathcal{C} \left(\fint_{\mathcal{A}_R} u^{p^*} \, dx \right)^{\frac{1}{p^*}} \qquad \forall R<\overline{R}/8 \text{ or } R>8/\overline{R},
	\end{equation} 
	where $\displaystyle \fint_{\mathcal{D}_R} u^t \, dx = \frac{1}{|\mathcal{D}_R|} \int_{\mathcal{D}_R} u^t \, dx$ and $\mathcal{C}$ is a positive constant  depending only on $N, p, \gamma, \overline R$ and $t$.
\end{lem}

\begin{proof}
	It is easy to see that, setting $\hat u(x)=u(Rx)$, for $R>0$, 
	$$-\Delta_p^H \hat u - \gamma \frac{\hat u^{p-1}}{H^\circ(x)^p}=R^p\hat u^{p^*-1} \qquad \text{in } \mathcal{A}_1.$$
	Let $m \geq 1$ and set
	\begin{equation}\label{eq:integralSet}
		\mathcal{A}:= \mathcal{A}_1 \cap \{\hat u < m\} \quad \text{and} \quad \mathcal{B} := \mathcal{A}_1 \cap \{\hat u \geq m\}.
	\end{equation}
	Hence we can consider the weak formulation of the last equation as follows
	\begin{equation}\label{eq:weakforspecial}
		 \begin{split}
			I_1+I_2:=& \int_{\mathcal{A}} H^{p-1}(\nabla \hat u) \langle \nabla H(\nabla \hat u) ,  \nabla \varphi  \rangle  \, dx + \int_{\mathcal{B}} H^{p-1}(\nabla \hat u) \langle \nabla H(\nabla \hat u) ,  \nabla \varphi  \rangle  \, dx \\
			=&\int_{\mathcal{A}_1} \gamma \frac{\hat u^{p-1}}{H^\circ(x)^p}  \varphi \, dx + R^p\int_{\mathcal{A}_1} 	\hat u^{p^*-1} \varphi \, dx \qquad \forall \varphi \in C^\infty_c(\mathcal{A}_1).
		\end{split}
	\end{equation}
	Let define $\hat u_m:=\min(\hat u, m)$ for $m \geq 1$. By density argument, for any $\eta \in C^\infty_c(\mathcal{A}_1)$ it is possible to choose $\varphi=\eta^p \hat u_m^{p(s-1)}\hat u$, with $s\ge 1$, as test function in \eqref{eq:weakforspecial}, so that, using \eqref{eq:integralSet} and \eqref{eq:EulerThm}, we can compute $I_1$ and $I_2$ 
	\begin{equation}\label{eq:MoserI_1}
		\begin{split}
			I_1 =& \int_{\mathcal{A}} H^{p-1}(\nabla \hat u) \langle \nabla H(\nabla \hat u) , p \eta^{p-1} \hat u^{p(s-1)+1} \nabla \eta  \rangle  \, dx \\
			&+ [p(s-1)+1] \int_{\mathcal{A}} H^{p-1}(\nabla \hat u) \langle \nabla H(\nabla \hat u) ,  \eta^p \hat u^{p(s-1)} \nabla \hat u \rangle  \, dx\\
			=& p\int_{\mathcal{A}} H^{p-1}(\nabla \hat u) \langle \nabla H(\nabla \hat u) , \nabla \eta  \rangle \eta^{p-1} \hat u^{p(s-1)+1}  \, dx \\
			&+ [p(s-1)+1] \int_{\mathcal{A}} H^p(\nabla \hat u)  \eta^p \hat u^{p(s-1)}   \, dx.\\
		\end{split}
	\end{equation}
In the same way, we obtain 
\begin{equation}\label{eq:MoserI_2}
	\begin{split}
		I_2 =& \int_{\mathcal{B}} H^{p-1}(\nabla \hat u) \langle \nabla H(\nabla \hat u) , p \eta^{p-1} m^{p(s-1)} \hat u \nabla \eta  \rangle  \, dx \\
		&+  \int_{\mathcal{B}} H^{p-1}(\nabla \hat u) \langle \nabla H(\nabla \hat u) ,  \eta^p m^{p(s-1)} \nabla \hat u \rangle  \, dx\\
		=& p\int_{\mathcal{B}} H^{p-1}(\nabla \hat u) \langle \nabla H(\nabla \hat u) , \nabla \eta  \rangle \eta^{p-1}  m^{p(s-1)} \hat u \, dx  + \int_{\mathcal{B}} H^p(\nabla \hat u)  \eta^p m^{p(s-1)}   \, dx.\\
	\end{split}
\end{equation}
Collecting both \eqref{eq:MoserI_1}, \eqref{eq:MoserI_2}, using the Scwharz's inequality and recalling \eqref{eq:BddH} we obtain
\begin{equation}\label{eq:MoserStep2}
	\begin{split}
		[p&(s-1)+1] \int_{\mathcal{A}} H^p(\nabla \hat u)  \eta^p \hat u^{p(s-1)}   \, dx +\int_{\mathcal{B}} H^p(\nabla \hat u)  \eta^p m^{p(s-1)}   \, dx\\
		\leq&  p M \left[ \int_{\mathcal{A}} H^{p-1}(\nabla \hat u) \eta^{p-1} \hat u^{p(s-1)+1} |\nabla \eta|  \, dx + \int_{\mathcal{B}} H^{p-1}(\nabla \hat u)   \eta^{p-1}  m^{p(s-1)} |\nabla \eta| \hat u \, dx\right] \\
		&+\int_{\mathcal{A}_1} \gamma \frac{\hat u^{p-1}}{H^\circ(x)^p}  \varphi \, dx + R^p\int_{\mathcal{A}_1} \hat u^{p^*-1} \varphi \, dx\\
		=&  p M \left[ \int_{\mathcal{A}} H^{p-1}(\nabla \hat u) \eta^{p-1} \hat u^{(p-1)(s-1)} \hat u^s |\nabla \eta|  \, dx + \int_{\mathcal{B}} H^{p-1}(\nabla \hat u)   \eta^{p-1}  m^{(p-1)(s-1)}m^{s-1} |\nabla \eta| \hat u \, dx \right]  \\
		&+ \int_{\mathcal{A}_1} \gamma \frac{\hat u^{p-1}}{H^\circ(x)^p}  \varphi \, dx + R^p\int_{\mathcal{A}_1} 	\hat u^{p^*-1} \varphi \, dx.
	\end{split}
\end{equation}
Now we can apply the weighted Young's inequality to the first two terms in the right hand side of \eqref{eq:MoserStep2} with conjugate exponent $\left(p/(p-1), p\right)$ in order to obtain
\begin{equation}\label{eq:MoserStep3}
	\begin{split}
		[p(s-1)&+1] \int_{\mathcal{A}} H^p(\nabla \hat u)  \eta^p \hat u^{p(s-1)}   \, dx +\int_{\mathcal{B}} H^p(\nabla \hat u)  \eta^p m^{p(s-1)}   \, dx\\
		\leq&   \varepsilon_1 [p(s-1)+1] \int_{\mathcal{A}} H^p(\nabla \hat u) \eta^p \hat u^{p(s-1)} \, dx + \mathcal{C}_{\varepsilon_1}(p,s,M) \int_{\mathcal{A}} \hat u^{ps} |\nabla \eta|^p  \, dx\\
		& + \varepsilon_2 \int_{\mathcal{B}} H^p(\nabla \hat u)   \eta^p  m^{p(s-1)} \,dx + \mathcal{C}_{\varepsilon_2}(p,M) \int_{\mathcal{B}}|\nabla \eta|^p  m^{p(s-1)} \hat u^p \, dx  \\
		&+ \int_{\mathcal{A}_1} \gamma \frac{\hat u^{p-1}}{H^\circ(x)^p}  \varphi \, dx + R^p\int_{\mathcal{A}_1} 	\hat u^{p^*-1} \varphi \, dx,
	\end{split}
\end{equation}
where $\mathcal{C}_{\varepsilon_1}(p,s,M)$ and $\mathcal{C}_{\varepsilon_2}(p,M)$ are two positive constants.
Hence we obtain 
\begin{equation}\label{eq:MoserStep4}
\begin{split}
	(1&-\varepsilon_1)[p(s-1)+1] \int_{\mathcal{A}} H^p(\eta \hat u_m^{s-1}\nabla \hat u)     \, dx +(1-\varepsilon_2)\int_{\mathcal{B}} H^p(\eta \hat u_m^{s-1}\nabla \hat u)     \, dx\\
	=&(1-\varepsilon_1)[p(s-1)+1] \int_{\mathcal{A}} H^p(\nabla \hat u)  \eta^p \hat u^{p(s-1)}   \, dx +(1-\varepsilon_2)\int_{\mathcal{B}} H^p(\nabla \hat u)  \eta^p m^{p(s-1)}   \, dx\\
	\leq&  \mathcal{C}_{\varepsilon_1}(p,s,M) \int_{\mathcal{A}}  |\nabla \eta|^p \hat u^{p(s-1)} \hat u^p \, dx + \mathcal{C}_{\varepsilon_2}(p,M) \int_{\mathcal{B}}|\nabla \eta|^p  m^{p(s-1)} \hat u^p \, dx  \\
	&+ \int_{\mathcal{A}_1} \frac{\gamma}{H^\circ(x)^p} \hat u^{p-1} \eta^p \hat u_m^{p(s-1)}\hat u \, dx + R^p\int_{\mathcal{A}_1} \hat u^{p^*-1} \eta^p \hat u_m^{p(s-1)}\hat u \, dx.
		\end{split}
	\end{equation}
Thanks to \eqref{eq:FinTrIn} and applying Lemma \ref{lem:a^p>b^p}, with $a=H(\eta \hat u^{s-1}\nabla \hat u)$ and $b=H(\hat u^s\nabla \eta)$, we deduce that there hold the following inequalities in the sets $\mathcal{A}$ and $\mathcal{B}$ respectively:
\begin{equation}\label{eq:lemmaTypeYoungBis1}
\begin{split}
	\int_{\mathcal{A}} H^p(\eta \hat u_m^{s-1}\nabla \hat u)  \,dx &=\int_{\mathcal{A}} H^p(\eta \hat u^{s-1}\nabla \hat u)  \,dx\\
	& \geq \frac{1}{1+2^{p+1}\delta_1} \cdot \frac{1}{s^p} \int_{\mathcal{A}} H^p(\nabla (\eta \hat u^s)) \,dx - \mathcal{C}_{\delta_1} \int_{\mathcal{A}} H^p(\hat u^s \nabla \eta)\, dx\\
	&= \frac{1}{1+2^{p+1}\delta_1} \cdot \frac{1}{s^p} \int_{\mathcal{A}} H^p(\nabla (\eta \hat u_m^{s-1} \hat u)) \,dx - \mathcal{C}_{\delta_1} \int_{\mathcal{A}} H^p(\hat u_m^{s-1} \hat u\nabla \eta ) \, dx
\end{split}
\end{equation}
\begin{equation}\label{eq:lemmaTypeYoungBis2}
	\begin{split}
	\int_{\mathcal{B}} H^p(\eta \hat u_m^{s-1}\nabla \hat u)  \,dx 
	& =	\int_{\mathcal{B}} H^p(\eta m^{s-1}\nabla \hat u)  \,dx\\ 
	&\geq \frac{1}{1+2^{p+1}\delta_2} \int_{\mathcal{B}} H^p(\nabla (\eta m^{s-1} \hat u)) \,dx - \mathcal{C}_{\delta_2} \int_{\mathcal{B}} H^p(m^{s-1} \hat u \nabla \eta) \, dx\\
	&= \frac{1}{1+2^{p+1}\delta_2} \int_{\mathcal{B}} H^p(\nabla (\eta \hat u_m^{s-1} \hat u)) \,dx - \mathcal{C}_{\delta_2} \int_{\mathcal{B}} H^p(\hat u_m^{s-1} \hat u\nabla \eta ) \, dx
\end{split}
\end{equation}
By \eqref{eq:lemmaTypeYoungBis1} and \eqref{eq:lemmaTypeYoungBis2}, we obtain 
\begin{equation}\label{eq:MoserStep5}
	\begin{split}
		&\frac{1-\varepsilon_1}{1+2^{p+1}\delta_1} \cdot \frac{p(s-1)+1}{s^p} \int_{\mathcal{A}}  H^p(\nabla (\eta \hat u_m^{s-1} \hat u))     \, dx +\frac{1-\varepsilon_2}{1 + 2^{p+1}\delta_2}\int_{\mathcal{B}}  H^p(\nabla (\eta \hat u_m^{s-1} \hat u))    \, dx\\
		\leq& \mathcal{C}_{\delta_1}(\varepsilon_1,p,s) \int_{\mathcal{A}}  H^p(\hat u_m^{s-1} \hat u\nabla \eta ) \, dx + \mathcal{C}_{\delta_2}(\varepsilon_2) \int_{\mathcal{B}}  H^p( \hat u_m^{s-1} \hat u\nabla \eta ) \, dx\\
		&+  \mathcal{C}_{\varepsilon_1}(p,s,M) \int_{\mathcal{A}}  |\nabla \eta|^p \hat u_m^{p(s-1)} \hat u^p \, dx + \mathcal{C}_{\varepsilon_2}(p,M) \int_{\mathcal{B}}|\nabla \eta|^p  \hat u_m^{p(s-1)} \hat u^p \, dx  \\
		&+ \int_{\mathcal{A}_1} \frac{\gamma}{H^\circ(x)^p} \hat u^{p-1} \eta^p \hat u_m^{p(s-1)}\hat u \, dx + R^p\int_{\mathcal{A}_1} 	\hat u^{p^*-p} \eta^p \hat u_m^{p(s-1)}\hat u^p \, dx.
	\end{split}
\end{equation}
Using \eqref{eq:equivalentNorm}, we deduce that
\begin{equation}\label{eq:MoserStep6}
	\begin{split}
		&\frac{1-\varepsilon_1}{1+2^{p+1}\delta_1} \cdot \frac{p(s-1)+1}{s^p} \int_{\mathcal{A}}  H^p(\nabla (\eta \hat u_m^{s-1} \hat u))   \, dx +\frac{1-\varepsilon_2}{1 + 2^{p+1}\delta_2} \int_{\mathcal{B}}  H^p(\nabla (\eta \hat u_m^{s-1} \hat u))   \, dx\\
		\leq&  \int_{\mathcal{A}_1} \frac{\gamma}{H^\circ(x)^p} \hat u^{p-1} \eta^p \hat u_m^{p(s-1)}\hat u \, dx + \hat{\mathcal{C}} \int_{\mathcal{A}_1}|\nabla \eta|^p  \hat u_m^{p(s-1)} \hat u^p \, dx  + R^p\int_{\mathcal{A}_1} \hat u^{p^*-p} (\eta \hat u_m^{s-1}\hat u)^p \, dx,
	\end{split}
\end{equation}
where $\hat{\mathcal{C}}$ depends on $\delta_1, \delta_2, \varepsilon_1, \varepsilon_2, p, s, M$. Using Hardy's and H\"older's inequality in the right hand side of \eqref{eq:MoserStep6}, and the definition of the sets $\mathcal{A}$ and $\mathcal{B}$, we obtain
\begin{equation}\label{eq:MoserStep7}
	\begin{split}
		&\frac{1-\varepsilon_1}{1+2^{p+1}\delta_1} \cdot \frac{p(s-1)+1}{s^p} \int_{\mathcal{A}}  H^p(\nabla (\eta \hat u_m^{s-1} \hat u))     \, dx +\frac{1-\varepsilon_2}{1 + 2^{p+1}\delta_2}\int_{\mathcal{B}}  H^p(\nabla (\eta \hat u_m^{s-1} \hat u))     \, dx\\
		\leq&  \frac{\gamma}{C_H}\int_{\mathcal{A}}  H^p(\nabla (\eta \hat u_m^{s-1} \hat u)) \, dx + \frac{\gamma}{C_H} \int_{\mathcal{B}}  H^p(\nabla (\eta \hat u_m^{s-1} \hat u)) \, dx\\
		&+ \hat{\mathcal{C}} \int_{\mathcal{A}_1}|\nabla \eta|^p  \hat u_m^{p(s-1)} \hat u^p \, dx  + \left(\int_{\mathcal{A}_1}  \hat u^{p^*}	R^N \,dx\right)^{\frac{p^*-p}{p^*}} \left(\int_{\mathcal{A}_1} (\eta \hat u_m^{s-1}\hat u)^{p \chi}\, dx\right)^{\frac 1\chi},
	\end{split}
\end{equation}
where $\chi=p^*/p$. Finally, we deduce
\begin{equation}\label{eq:MoserStep8}
\begin{split}
	&\left(\frac{1-\varepsilon_1}{1+2^{p+1}\delta_1} \cdot  \frac{p(s-1)+1}{s^p} - \frac{\gamma}{C_H}\right) \int_{\mathcal{A}}  H^p(\nabla (\eta \hat u_m^{s-1} \hat u))     \, dx \\
	& \qquad +\left(\frac{1-\varepsilon_2}{1 + 2^{p+1}\delta_2} - \frac{\gamma}{C_H}\right)\int_{\mathcal{B}}  H^p(\nabla (\eta \hat u_m^{s-1} \hat u))  \, dx\\
	& \; \leq  \hat{\mathcal{C}} \int_{\mathcal{A}_1}|\nabla \eta|^p  \hat u_m^{p(s-1)} \hat u^p \, dx  +  \| u \|^{p^*-p}_{L^{p^*}(\mathcal{A}_R)} \left(\int_{\mathcal{A}_1} (\eta \hat u_m^{s-1}\hat u)^{p \chi}\, dx\right)^{\frac{1}{\chi}}.
\end{split}
\end{equation}
Noticing that $(p(s-1)+1)/s^p > \gamma/C_H$ for all $s \in \left( (N-p)/(p \mu_2), (N-p)(p \mu_1)\right)$, we can fix $\delta_1, \varepsilon_1>0$ sufficiently small such that 
$$\frac{1-\varepsilon_1}{1+2^{p+1}\delta_1} \cdot  \frac{p(s-1)+1}{s^p} - \frac{\gamma}{C_H}>0,$$
and $\delta_2, \varepsilon_2>0$ sufficiently small such that $$\frac{1-\varepsilon_2}{1 + 2^{p+1}\delta_2} - \frac{\gamma}{C_H} > 0.$$
Hence by \eqref{eq:MoserStep8} we get
\begin{equation}\label{eq:MoserStep9}
	\bar{\mathcal{C}} \int_{\mathcal{A}_1}  H^p(\nabla (\eta \hat u_m^{s-1} \hat u))  \, dx \leq  \hat{\mathcal{C}} \int_{\mathcal{A}_1}|\nabla \eta|^p  \hat u_m^{p(s-1)} \hat u^p \, dx  + \| u \|^{p^*-p}_{L^{p^*}(\mathcal{A}_R)} \left(\int_{\mathcal{A}_1} (\eta \hat u_m^{s-1}\hat u)^{p \chi}\, dx\right)^{\frac{1}{\chi}},
\end{equation}
where $\bar{\mathcal{C}}$ is a positive constant depending on $\varepsilon_1,\varepsilon_2,\delta_1,\delta_2, p,s,\gamma,C_H$.
In conclusion, by \eqref{eq:equivalentNorm} and Sobolev inequality we have
\begin{equation}\label{eq:MoserStep10}
\begin{split}
	\alpha_1^p \bar{\mathcal{C}} \mathcal{C}_S^p\left(\int_{\mathcal{A}_1} (\eta \hat u_m^{s-1}\hat u)^{p \chi}\, dx \right)^{\frac{1}{\chi}}  \leq & \alpha_1^p \bar{\mathcal{C}} \int_{\mathcal{A}_1} |\nabla (\eta \hat u_m^{s-1} \hat u)|^p  \, dx  \leq \bar{\mathcal{C}} \int_{\mathcal{A}_1}  H^p(\nabla (\eta \hat u_m^{s-1} \hat u))  \, dx \\
	\leq &  \hat{\mathcal{C}} \int_{\mathcal{A}_1}|\nabla \eta|^p  \hat u_m^{p(s-1)} \hat u^p \, dx   \\ & +  \| u \|^{p^*-p}_{L^{p^*}(\mathcal{A}_R)} \left(\int_{\mathcal{A}_1} (\eta \hat u_m^{s-1}\hat u)^{p \chi}\, dx \right)^{\frac{1}{\chi}}.
\end{split}
\end{equation}
In order to apply the Moser's iteration method we need to rewrite the last inequality as follows
\begin{equation}\label{eq:MoserStep11}
	\left(\int_{\mathcal{A}_1} (\eta \hat u_m^{s-1}\hat u)^{p \chi}\, dx \right)^{\frac{1}{\chi}} \leq  \mathcal{C}_1 \int_{\mathcal{A}_1}|\nabla \eta|^p  \hat u_m^{p(s-1)} \hat u^p \, dx  +  \mathcal{C}_2 \| u \|^{p^*-p}_{L^{p^*}(\mathcal{A}_R)} \left(\int_{\mathcal{A}_1} (\eta \hat u_m^{s-1}\hat u)^{p \chi}\, dx \right)^{\frac{1}{\chi}},
\end{equation}
where $\mathcal{C}_1:=\mathcal{\hat C}/(\alpha_1^p \bar{\mathcal{C}} \mathcal{C}_S^p)$,  $\mathcal{C}_2:=1/(\alpha_1^p \bar{\mathcal{C}} \mathcal{C}_S^p)$    and $\chi={p^*}/{p}$.

Fix $t\in (p^*,N/\mu_1)$ and $k\in \N$ so that $p\chi^k\le t\le p\chi ^{k+1}$. Then there exist positive constants $\mathcal C_1$ and $\mathcal C_2$ such that \eqref{eq:MoserStep11} holds for all $1\le s\le \min \left\{(N-p)/(p\mu_1),\chi^k\right\}$.  Now, we set  $\sigma
=(1 / (2\mathcal{C}_2))^\frac{1}{p^*-p}$ and choosing $\overline R$ sufficiently small such that \eqref{eq:assumptionsmallaux2} holds, we get 
\begin{equation}
	\left(\int_{\mathcal{A}_1} (\eta \hat u_m^{s-1}\hat u)^{p \chi}\, dx \right)^{\frac{1}{\chi}} \leq  \mathcal{C} \int_{\mathcal{A}_1}|\nabla \eta|^p  \hat u_m^{p(s-1)} \hat u^p \, dx  
\end{equation}
for all $1\le s\le \min \left\{(N-p)/(p\mu_1),\chi^k\right\}$ and $\mathcal C=\mathcal C_1/2.$ Applying Moser's iteration method (see e.g. \cite{HanLin,Xiang} for further details), we conclude, after finitely many times of iteration, 
\begin{equation}\label{abc}
	\left(\fint_{\mathcal{D}_1}  |\hat u|^t\, dx \right)^{\frac{1}{t}} \leq  \mathcal{C} \left(\fint_{\mathcal{A}_1}| \hat u|^{p^*} \, dx \right) ^{\frac{1}{p^*}}.
\end{equation}
for any $t\in (p^*,N/\mu_1).$
This proves \eqref{eq:EstAt0aux2}.

\end{proof}

Let us now prove the following:
\begin{thm}\label{teoo}
	Let $u$ be a weak solution of \eqref{eq:ProbDoublyCritical}. Then there exists a positive constant $C=C(N,p,\gamma,u)$ such that
	$$|u(x)|\leq \frac{C}{[H^\circ(x)]^{\frac{N-p}{p}-\sigma_1}} \qquad \text{in } \mathcal{B}_{R_1}^{H^\circ},$$
	and that
	$$|u(x)|\leq \frac{C}{{[H^\circ(x)]^{\frac{N-p}{p}+\sigma_2}}} \qquad \text{in } (\mathcal{B}_{R_2}^{H^\circ})^c,$$
	where $\sigma_1, \sigma_2$ are givem in Lemma \ref{lem:aux1} and  $ R_1, R_2 >0$ are constants depending on $N, p, \gamma$ and $u$.
\end{thm}

\begin{proof}
Let us fix $t := (p^*+ N/\mu_1)/{2} \in \left(p^*, N/\mu_1\right)$ as in Lemma \ref{aux2} and $\kappa := \min\{\tau, \sigma\}$, where $\tau$ and $\sigma$ are respectively as in Lemma \ref{lem:aux1} and Lemma \ref{aux2}. Let $\bar R > 0$ such that \eqref{eq:assumptionsmall} holds for $\kappa$ and let us consider
$\hat u(x)=u(Rx)$, for $R>0$ fixed. We note that $\hat u$ satisfies the equation
$$-\Delta_p^H \hat u + c(x) \hat u^{p-1}=0 \qquad \text{in } \mathcal{D}_1,$$
where $$c(x)=-\frac{\gamma}{H^\circ(x)^p}-R^p\hat u ^{p^*-p}(x).$$
We note that $\displaystyle H^\circ(x)^{-p}$ is bounded in $\mathcal{D}_1$ and $V_R(x):=R^p\hat u ^{p^*-p}(x)\in L^q(\mathcal D_1)$ with $q=t/(p^*-p)>N/p$ due to Lemma \ref{aux2}. Hence, as in the proof of \cite[Theorem $1$]{serrin} a classical Moser iteration argument yields 
\begin{equation}\label{eq:MoserHanLin}
	\sup_{\mathcal{B}^{H^\circ}_r(x)} \hat u \leq \mathbf{C} \left(\fint_{\mathcal{B}^{H^\circ}_{2r}(x)} \hat u^p \, dx\right)^\frac{1}{p}
\end{equation}
for any ball $\mathcal{B}^{H^\circ}_{2r}(x) \subset \mathcal{D}_1$, where $\mathbf{C}=\mathbf{C}(N,p,\gamma,\|V_R\|_{L^q(\mathcal D_1)})$. We claim that $\|V_R\|_{L^q(\mathcal D_1)}$ is uniformly bounded with respect to $R$. Indeed from Lemma \ref{aux2}, since 
$$p-\frac{N}{q}+N\frac{p^*-p}{t}-N\frac{p^*-p}{p^*}=0,$$
we get

\begin{equation}\label{cnondipendedaR}
    \begin{split}
        \left( \int_{\mathcal D_1}V_R^q \, dx\right)^{\frac{1}{q}}& = R^{p-\frac{N}{q}}\left( \int_{\mathcal D_R}  u^{(p^*-p)q} \, dx\right)^{\frac{1}{q}} \\ &\le \mathcal CR^{p-\frac{N}{q}+N\frac{p^*-p}{t}}\left( \fint_{\mathcal D_R}  u^{t} \, dx\right)^{\frac{p^*-p}{t}} \\ &\le \mathcal C R^{p-\frac{N}{q}+N\frac{p^*-p}{t}-N\frac{p^*-p}{p^*}}\left( \int_{\mathcal A_R}  u^t \, dx \right)^{\frac{p^*-p}{p^*}} \\ &\le \mathcal C \left(\int_{\R^N}u^{p^*} \, dx\right)^{\frac{p^*-p}{p^*}} \qquad \forall 0<R\le \frac{\overline R}{8}\text{ or } R\ge \frac{8}{\overline R},
    \end{split}
\end{equation}
where $\mathcal C$ is a positive constant depending on $N, p, \gamma, q$ and $\overline R$.

Using a covering argument we deduce that
\begin{equation}\label{eq:MoserHanLin2}
	\sup_{\mathcal{B}^{H^\circ}_2 \setminus \mathcal{B}^{H^\circ}_1} \hat u \leq \mathbf{C} \left(\fint_{\mathcal{D}_1} \hat u^p \, dx\right)^\frac{1}{p}
\end{equation}
Noticing that $\hat u(x)= u(Rx)$, by \eqref{eq:MoserHanLin2} we obtain that
\begin{equation}\label{eq:MoserHanLin3}
	\sup_{\mathcal{B}^{H^\circ}_{2R} \setminus \mathcal{B}^{H^\circ}_R} u \leq \mathbf{C} \left(\fint_{\mathcal{D}_R}  u^p \, dx\right)^\frac{1}{p}
\end{equation}
for each $0 < R \leq \bar R/8$ or $R \geq 8/\bar R$. By applying the H\"older's inequality in \eqref{eq:MoserHanLin3}, we get
\begin{equation}\label{eq:MoserHanLin4}
	\sup_{\mathcal{B}^{H^\circ}_{2R} \setminus \mathcal{B}^{H^\circ}_R} u \leq \mathbf{C} \left(\fint_{\mathcal{D}_R}  u^p \, dx \right)^\frac{1}{p} \leq \mathbf{C} \left(\fint_{\mathcal{D}_R}  u^{p^*} \, dx\right)^\frac{1}{p^*} = \mathbf{C} \|u\|_{L^{p^*}(\mathcal{D}_R)} R^{\frac{p-N}{p}},
\end{equation}
for each $0< R \leq \bar R/8$ or $R \geq 8/\bar R$ and $\mathbf{C}$ depends only on $N, p, \gamma, q, \bar R$ and $\| u \|_{L^{p^*} (\R^N)}$.

Now we note that, since $\mathcal{A}_R \subset \mathcal{B}^{H^\circ}_{\bar R}$ for any $0<R\leq \bar R/8$ and  $\mathcal{A}_R \subset (\mathcal{B}^{H^\circ}_{1/\bar R})^c$ for any $R \geq 8/\bar R$, there exist, by Lemma \ref{lem:aux1}, $\sigma_1, \sigma_2>0$ depending only on $N, p, \gamma$ such that
$$\|u\|_{L^{p^*}(\mathcal{B}^{H^\circ}_R)} \leq \mathcal{C} R^{\sigma_1}  \quad \text{for } 0 < R \leq \frac{\bar R}{8}$$
and that
$$\|u\|_{L^{p^*}((\mathcal{B}^{H^\circ}_R)^c)} \leq \frac{\mathcal{C}}{R^{\sigma_2}} \quad \text{for } R \geq 8/\bar R.$$ 
Now, if we set $R_1= \bar R/8$ and $R_2=8/\bar R$, by \eqref{eq:MoserHanLin4} we get the thesis.

\end{proof}

The next result is devoted to show the existence of some special supersolutions of our problem, in order to perform a comparison between them and the solutions of the doubly critical equation \eqref{eq:ProbDoublyCritical}.

\begin{prop} \label{prop:SubSuperComp}
	Given two constants $A>0$ and $\alpha < p$, there exist constants $0<\varepsilon,\delta<1$, depending on $N,p,\gamma,A,\alpha$, such that  
	\begin{equation}v(H^\circ(x))=\frac{1-\delta [H^\circ(x)]^\varepsilon}{[H^\circ(x)]^{\mu_1}} \in \mathcal{D}^{1,p}(\mathcal{B}^{H^\circ}_{R_1})
 \end{equation}

	is a positive supersolution to equation
	\begin{equation}-\Delta_p^H v - \frac{\gamma}{H^\circ(x)^p} v^{p-1} = g(x) v^{p-1} \quad \text{in } \mathcal{B}^{H^\circ}_{R_1},\end{equation}
	for some positive constant $0<R_1<1$ depending only on $N,p,\gamma, A$ and $\alpha$, where $g(x)$ is a positive function that belongs to $L^{\frac{N}{p}}(\mathcal{B}^{H^\circ}_{R_1})$ such that
	\begin{equation}\label{g1}
 g(x) \geq A[H^\circ(x)]^{-\alpha}\quad\text{in }\mathcal{B}^{H^\circ}_{R_1} .
 \end{equation}
	
	In a similar way, given $A>0$ and $\alpha > p$, there exist $0<\varepsilon, \delta <1$  such that
	\begin{equation}v(H^\circ(x))=\frac{1-\delta [H^\circ(x)]^{-\varepsilon}}{[H^\circ(x)]^{\mu_2}} \in \mathcal{D}^{1,p}((\mathcal{B}^{H^\circ}_{R_2})^c)\end{equation}
	is a positive supersolution to equation
	\begin{equation}-\Delta_p^H v - \frac{\gamma}{H^\circ(x)^p} v^{p-1} = g(x)v^{p-1} \quad \text{in } (\mathcal{B}^{H^\circ}_{R_2})^c,\end{equation}
	for some positive constant $R_2>1$ depending only on $N,p,\gamma$ and $\alpha$,  where $g(x)$ is a positive function that belongs to $L^{\frac{N}{p}}((\mathcal{B}^{H^\circ}_{R_2})^c)$ such that
	\begin{equation}\label{g2}
 g(x) \geq A[H^\circ(x)]^{-\alpha}\quad \text{in } (\mathcal{B}^{H^\circ}_{R_2})^c.
 \end{equation}
\end{prop}

\begin{proof}
Let us consider $\mu, \delta, \varepsilon>0$ and let us define the function
$$u(x)=v(H^\circ(x))=\frac{1-\delta [H^\circ(x)]^\varepsilon}{[H^\circ(x)]^\mu}.$$
It is easy to deduce that
\begin{equation}
	\nabla u (x) = s(H^\circ(x)) \nabla H ^\circ(x),
\end{equation}
where $s(t):=t^{-\mu-1}[-\mu + \delta(\mu-\varepsilon)t^\varepsilon]$. Using \eqref{eq:PropFinsler}, we now compute
\begin{equation} \label{eq:computedSuperBDD}
\begin{split}
	-\Delta_p^H u =& -\operatorname{div} \left(H^{p-1} (\nabla u) \nabla H (\nabla u) \right) = - \operatorname{div} \left(|s(H^\circ(x))|^{p-1} \operatorname{sign}(s) \nabla H (\nabla H^\circ(x))\right) \\
	=& -(p-1) |s(H^\circ(x))|^{p-2} s'(H^\circ(x)) \langle \nabla H^\circ(x) , \nabla H(\nabla H^\circ(x)) \rangle\\
	&- |s(H^\circ(x))|^{p-1} \operatorname{sign}(s) \operatorname{div} \left(\nabla H (\nabla H^\circ(x))\right)\\
	=& -(p-1) |s(H^\circ(x))|^{p-2} s'(H^\circ(x)) - (N-1) \frac{|s(H^\circ(x))|^{p-2}s(H^\circ(x)) }{H^\circ(x)},
\end{split} 
\end{equation}
where in the last line we used the fact that $\langle \nabla H^\circ(x) , \nabla H(\nabla H^\circ(x)) \rangle =1$ and $$\displaystyle \operatorname{div}\left(\nabla H (\nabla H^\circ(x))\right) = \frac{N-1}{H^\circ(x)}$$
due to \eqref{eq:PropFinsler} and \eqref{eq:PropFinsler2}.

Making standard computations on the right hand side of \eqref{eq:computedSuperBDD}, one can deduce
\begin{equation} \label{eq:computedSuperBDD2}
	\begin{split}
		-\Delta_p^H u -\frac{\gamma}{[H^\circ(x)]^p} |u|^{p-2} u= g(H^\circ(x)) |u|^{p-2} u,
	\end{split} 
\end{equation}
where
$$g(t)= \frac{h(t)}{|1-\delta t^\varepsilon|^{p-2} (1-\delta t^\varepsilon) t^p},$$
and
\begin{equation}
\begin{split}
h(t):=&-\left|\mu - \delta (\mu - \varepsilon) t^\varepsilon\right|^{p-2} \left\{\mu^2(p-1) - (N-p) \mu - (p-1)(\mu-\varepsilon)^2 \delta t^\varepsilon + (N-p) (\mu - \varepsilon)\delta t^{\varepsilon} \right\} \\
&- \gamma |1-\delta t^\varepsilon|^{p-2}(1-\delta t^\varepsilon), \qquad \text{for } t \in [0,+\infty).
\end{split}
\end{equation}
We note that $h(0)= - |\mu|^{p-2}[\mu^2(p-1)-\mu(N-p)]-\gamma$.
Hence, using the definition of $\mu_1$ and $\mu_2$, we deduce that $h(0)=0$ when $\mu=\mu_1$ and $\mu=\mu_2$. Also we have $h'(0)>0$ if $\mu =\mu_1$, $\varepsilon>0$ or $\mu=\mu_2$, $\varepsilon<0$. This implies there exist $0<\delta_h<1$ such that \begin{equation}\label{basta}
2h'(0)t\ge h(t)\ge \frac{1}{2}h'(0)t>0 \qquad \forall t\in(0,\delta_h).
\end{equation}We set $\delta=\min \{\delta_h, 1/2\}$, $\varepsilon=(p-\alpha)/2$ and $$R_0=\left\{1,\left(\frac{\delta}{2A}h'(0)\right)^{1/(p-\alpha-\varepsilon)}\right\}.$$ 
It easy to check that $v(H^\circ(x))=(1-\delta [H^\circ(x)]^\varepsilon)[H^\circ(x)]^{-\mu_1} \in \mathcal{D}^{1,p}(\mathcal{B}^{H^\circ}_{R_1})$ is positive, which thanks to \eqref{basta}, $g(x)\in L^{\frac{N}{p}}(\mathcal{B}^{H^\circ}_{R_1}) $ and it satisfies \eqref{g1}. The other case is similar. 

\end{proof}

Now, we consider the following equation
\begin{equation}\label{eq:comparison}
	-\Delta_p^H w - \frac{\gamma}{H^\circ(x)^p}w^{p-1}=f(x) w^{p-1} \qquad \text{in } \Omega,
\end{equation}
where $\Omega$ is an open subset of $\R^N$, $w>0$ and $w \in \mathcal{D}^{1,p}(\Omega)$.
Let us start with a comparison principle in bounded domains.

The first result is given by the following pointwise estimate, in the same spirit of \cite{OSV, Xiang}.

\begin{prop}\label{prop:Comparison1}
Let $u,v$ two weakly differentiable strictly positive functions on a domain~$\Omega$\footnote{We mean that $u,v \geq C > 0$ a.e. in $\Omega$.}. We have that:
\begin{itemize}
	\item[\textbf{(i)}] if $p \geq 2$, then
\begin{equation}\label{eq:compaeisonDegenerate}
\begin{split}
	&H^{p-1}(\nabla u) \langle \nabla H (\nabla u),  \nabla \left(u - \frac{v^p}{u^p} u \right) \rangle + H^{p-1}(\nabla v) \langle \nabla H (\nabla v), \nabla \left(v - \frac{u^p}{v^p} v \right) \rangle \\
	& \geq C_{p}(u^p+v^p) H^p(\nabla (\ln u -\ln v)),
\end{split}
\end{equation}
for some positive constant $C_{p}$ depending only on $p$;

\

\item[\textbf{(ii)}] if $1<p<2$, then
\begin{equation}\label{eq:comparisonSingular}
	\begin{split}
		&H^{p-1}(\nabla u) \langle \nabla H (\nabla u),  \nabla \left(u - \frac{v^p}{u^p} u \right) \rangle + H^{p-1}(\nabla v) \langle \nabla H (\nabla v), \nabla \left(v - \frac{u^p}{v^p} v \right)\rangle \\
		& \geq C_p (u^p+v^p)\left[H(\nabla \ln u) + H(\nabla\ln v)\right]^{p-2}H^2(\nabla (\ln u -\ln v)),
	\end{split}
\end{equation}
for some positive constant $C_{p}$ depending only on $p$.
\end{itemize}
\end{prop}

\begin{proof}
Let $u,v$ two weakly differentiable positive functions and consider the following 
$$\psi_1:=u-\frac{v^p}{u^p}u \qquad \text{and} \qquad \psi_2:=v-\frac{u^p}{v^p}v.$$
Then, thanks to the Euler's theorem for 1-homogeneous functions, and since $\nabla H$ is 0-homogeneus, we deduce that 
\begin{equation}\label{eq:estimate1Finsl}
\begin{split}
H^{p-1}(\nabla u) \langle \nabla & H (\nabla u), \nabla \psi_1 \rangle =\\
&= H^p(\nabla u) - p \frac{v^{p-1}}{u^{p-1}} H^{p-1}(\nabla u) \langle \nabla H (\nabla u), \nabla v \rangle + (p-1) \frac{v^p}{u^p} H^p(\nabla u)\\
&= H^p(\nabla u) - v^p \left[(1-p)H^p\left(\frac{\nabla u}{u}\right) + p H^{p-1} \left(\frac{\nabla u}{u}\right)\langle \nabla H(\nabla u), \frac{\nabla v}{v} \rangle\right]\\
&= H^p(\nabla u) - v^p \left[H^p(\nabla \ln u) + p H^{p-1} (\nabla \ln u)\langle \nabla H(\nabla \ln u), \nabla \ln v - \nabla \ln u \rangle\right]
\end{split}
\end{equation}
and
\begin{equation}\label{eq:estimate2Finsl}
\begin{split}
H^{p-1}(\nabla v) \langle \nabla &H (\nabla v),  \nabla \psi_2 \rangle = \\
&= H^p(\nabla v) - p \frac{u^{p-1}}{v^{p-1}} H^{p-1}(\nabla v) \langle \nabla H (\nabla v), \nabla u \rangle + (p-1) \frac{u^p}{v^p} H^p(\nabla v )\\
&= H^p(\nabla v) - u^p \left[(1-p)H^p\left(\frac{\nabla v}{v}\right) + p H^{p-1} \left(\frac{\nabla v}{v}\right)\langle \nabla H(\nabla v), \frac{\nabla u}{u} \rangle\right]\\
&= H^p(\nabla v) - u^p \left[H^p (\nabla \ln v) + p H^{p-1} (\nabla \ln v)\langle \nabla H(\nabla \ln v), \nabla \ln u - \nabla \ln v \rangle\right].
\end{split}
\end{equation}
\textbf{(i) $p \geq 2$.} We recall that when $p \geq 2$, it holds \eqref{introeq:ineqStandard}, i.e.
\begin{equation} \label{eq:ineqStandard}
	H^p(\eta) \geq H^p(\eta') + p H^{p-1}(\eta') \langle \nabla H(\eta'), \eta -\eta' \rangle + \frac{1}{2^{p-1}-1} H^p(\eta - \eta'), \qquad \forall \eta, \eta' \in \R^N.
\end{equation}
Hence we can apply this inequality, in order to give an estimate from below for \eqref{eq:estimate1Finsl} and \eqref{eq:estimate2Finsl}:
\begin{equation}\label{eq:estimate3Finsl}
	\begin{split}
		H^{p-1}(\nabla u) \langle \nabla H &(\nabla u), \nabla \psi_1 \rangle =\\
		&= H^p(\nabla u) - v^p \left[H^p(\nabla \ln u) + p H^{p-1} (\nabla \ln u)\langle \nabla H(\nabla \ln u), \nabla \ln v - \nabla \ln u \rangle\right]\\
		& \geq H^p(\nabla u) - v^p \left[H^p(\nabla \ln v) - C_p H^p(\nabla(\ln v - \ln u))\right]\\
		& = H^p(\nabla u) -  H^p(\nabla v) + C_p v^p H^p(\nabla(\ln v - \ln u)),
	\end{split}
\end{equation}
\begin{equation}\label{eq:estimate4Finsl}
	\begin{split}
		H^{p-1}(\nabla v) \langle \nabla H &(\nabla v), \nabla \psi_2 \rangle =\\
		&= H^p(\nabla v) - u^p \left[H^p(\nabla \ln v) + p H^{p-1} (\nabla \ln v)\langle \nabla H(\nabla \ln v), \nabla \ln u - \nabla \ln v \rangle\right]\\
		& \geq H^p(\nabla v) - u^p \left[H^p(\nabla \ln u) - C_p H^p(\nabla(\ln u - \ln v))\right]\\
		& = H^p(\nabla v) -  H^p(\nabla u) + C_p u^pH^p(\nabla(\ln u - \ln v)).
	\end{split}
\end{equation}
Adding both these two inequalities, we obtain
\begin{equation}\label{eq:estimate5Finsl}
	\begin{split}
		H^{p-1}(\nabla u) \langle \nabla H (\nabla u), \nabla \psi_1 \rangle + H^{p-1}(\nabla v) \langle \nabla H (\nabla v), \nabla \psi_2 \rangle \geq C_p(u^p+v^p) H^p(\nabla(\ln u - \ln v)).
	\end{split}
\end{equation}

\textbf{(ii) $1 < p < 2$.} We recall that when $1 < p < 2$, it holds \eqref{intro:eq:ineqStandardBis}, i.e.
\begin{equation} \label{eq:ineqStandardBis}
	H^p(\eta) \geq H^p(\eta') + p H^{p-1}(\eta') \langle \nabla H (\eta'), \eta-\eta' \rangle + C_p [H(\eta) + H(\eta')]^{p-2}H^2(\eta - \eta'), \quad \forall \eta, \eta' \in \R^N,
\end{equation}
where $C_p$ is a positive constant depending only on $p$.

Now we proceed exactly as in the previous case to get an estimate from below for \eqref{eq:estimate1Finsl} and \eqref{eq:estimate2Finsl}:
\begin{equation}\label{eq:estimate6Finsl}
	\begin{split}
		H^{p-1}(\nabla u) &\langle \nabla H (\nabla u), \nabla \psi_1 \rangle =\\
		&= H^p(\nabla u) - v^p \left[H^p(\nabla \ln u) + p H^{p-1} (\nabla \ln u)\langle \nabla H(\nabla \ln u), \nabla \ln v - \nabla \ln u \rangle\right]\\
		& \geq H^p(\nabla u) - v^p \left[H^p(\nabla \ln v) - C_p [H(\nabla \ln u) + H(\nabla \ln v)]^{p-2}H^2(\nabla(\ln v - \ln u))\right]\\
		& = H^p(\nabla u) -  H^p(\nabla v) + C_p v^p \left[[H(\nabla \ln u) + H(\nabla \ln v)]^{p-2} H^2(\nabla(\ln v - \ln u)\right],
	\end{split}
\end{equation}
\begin{equation}\label{eq:estimate7Finsl}
	\begin{split}
		H^{p-1}(\nabla v) &\langle \nabla H (\nabla v), \nabla \psi_2 \rangle =\\
		&= H^p(\nabla v) - u^p \left[H^p(\nabla \ln v) + p H^{p-1} (\nabla \ln v)\langle \nabla H(\nabla \ln v), \nabla \ln u - \nabla \ln v \rangle\right]\\
		& \geq H^p(\nabla v) - u^p \left[H^p(\nabla \ln u) - C_p [H(\nabla \ln u) + H(\nabla \ln v)]^{p-2}H^2(\nabla(\ln v - \ln u))\right]\\
		& = H^p(\nabla v) -  H^p(\nabla u) +C_p u^p[H(\nabla \ln u) + H(\nabla \ln v)]^{p-2}H^2(\nabla(\ln v - \ln u)).
	\end{split}
\end{equation}
Adding both these two inequalities, we obtain
\begin{equation}\label{eq:estimate8Finsl}
	\begin{split}
		H^{p-1}(\nabla u) \langle \nabla H (\nabla u), \nabla \psi_1 \rangle &+ H^{p-1}(\nabla v) \langle \nabla H (\nabla v), \nabla \psi_2 \rangle \\
		&\geq C_p (u^p+v^p)[H(\nabla \ln u) + H(\nabla \ln v)]^{p-2}H^2(\nabla(\ln v - \ln u)).
	\end{split}
\end{equation}

\end{proof}

Now, we are ready to prove the comparison principles in bounded and exteriors domains.

\begin{prop} \label{prop:Comparison2}
Let $\Omega$ be an open bounded smooth domain of $\R^N$ and $f\in L^{\frac{N}{p}}(\Omega)$. Let $u \in \mathcal {D} ^{1,p}(\Omega)$ be a weak positive subsolution to \eqref{eq:comparison} and $v\in \mathcal D^{1,p}(\Omega)$ be a weak positive supersolution of 
\begin{equation}\label{eq:comparison2}
	-\Delta_p^H v - \frac{\gamma}{H^\circ(x)^p}v^{p-1}=g(x) v^{p-1} \qquad \text{in } \Omega,
\end{equation}
with $g\in L^{\frac{N}{p}}(\Omega)$.
Assume that $\inf_\Omega v >0$ and $f\le g$ in $\Omega.$
If $u \leq v$ on $\partial \Omega$, then $$u \leq v \quad \text{in } \Omega.$$ 
\end{prop}

\begin{proof}
	We will give the proof of this result in the case $p \geq 2$. The case $1<p<2$ is similar.
	
	Let us define
	$$ \eta_1:=\frac{\min\{(u^p-v^p)^+,m\}}{u^{p-1}} \qquad \text{and} \qquad \eta_2:=\frac{	\min\{(u^p-v^p)^+,m\}}{v^{p-1}}$$
	where $m>1$. It is quite standard to show that $\eta_1$ and $\eta_2$ are good test function that we can use in the weak formulations of \eqref{eq:comparison} and \eqref{eq:comparison2}. Taking both these test function and subtracting the two equations, we obtain
	\begin{equation}
	\begin{split}
		\int_\Omega H^{p-1}(\nabla u) \langle \nabla H (\nabla u), \nabla \eta_1 \rangle \, dx &- \int_\Omega H^{p-1}(\nabla v) \langle \nabla H (\nabla v), \nabla \eta_2 \rangle \, dx\\
		\le & \int_\Omega \frac{\gamma}{H^\circ(x)^p}u^{p-1} \eta_1 \, dx - \int_\Omega \frac{\gamma}{H^\circ(x)^p}v^{p-1} \eta_2 \, dx \\
		&+ \int_\Omega f(x)u^{p-1} \eta_1 \,dx - \int_\Omega g(x) v^{p-1} \eta_2 \, dx\\
		=& \int_\Omega \frac{\gamma}{H^\circ(x)^p} \left(u^{p-1}\eta_1-v^{p-1}\eta_2\right) \, dx\\
		&+ \int_\Omega f(x) u^{p-1} \eta_1 -g(x)v^{p-1} \eta_2 \,dx \le 0,
	\end{split}
	\end{equation}
since $f\le g$ in $\Omega$.
Hence, setting  $\Omega_1 := \{x \in \Omega \ | \ 0 \leq u^p-v^p \leq m\}$ and $\Omega_2:= \{x \in \Omega \ | \ u^p-v^p \geq m\}$, we deduce that
\begin{equation}\label{eq:computationsComparison1}
	\begin{split}
		&\int_{\Omega_1} H^{p-1}(\nabla u) \langle \nabla H (\nabla u), \nabla \left(u-\frac{v^p}{u^p}u \right) \rangle \, dx +\int_{\Omega_1} H^{p-1}(\nabla v) \langle \nabla H (\nabla v), \nabla \left(v-\frac{u^p}{v^p}v \right)\rangle \, dx \\
		&+ m(1-p) \left[ \int_{\Omega_2} H^{p-1}(\nabla u) \langle \nabla H (\nabla u), u^{-p}\nabla u \rangle \, dx - \int_{\Omega_2} H^{p-1}(\nabla v) \langle \nabla H (\nabla v), v^{-p}\nabla v\rangle \, dx \right]\le 0.
	\end{split}
\end{equation}
Applying \eqref{eq:compaeisonDegenerate} in \eqref{eq:computationsComparison1} and making some computations we obtain
\begin{equation}\label{eq:computationsComparison2}
	\begin{split}
		C_p &\int_{\Omega_1} (u^p+v^p)H^p(\nabla (\ln u -\ln v)) \, dx \\
		&+ m(p-1) \left[\int_{\Omega_2} H^{p}(\nabla \ln v) \, dx - \int_{\Omega_2}  H^{p}(\nabla \ln u) \, dx \right]\leq 0,
	\end{split}
\end{equation}
but this implies that
\begin{equation}\label{eq:computationsComparison3}
	\begin{split}
		C_p \int_{\Omega_1} (u^p+v^p)H^p(\nabla (\ln u -\ln v)) \, dx \leq m(p-1) \int_{\Omega_2}  H^{p}(\nabla \ln u) \, dx.
	\end{split}
\end{equation}
For the right hand side of \eqref{eq:computationsComparison3}, we have
\begin{equation}\label{eq:RHScomputationsComparison3}
	\begin{split}
	m(p-1) \int_{\Omega_2}  H^{p}(\nabla \ln u) \, dx = (p-1)	\int_{\Omega_2} mu^{-p} H^p(\nabla u) \, dx \leq (p-1)\int_{\Omega_2'} H^p(\nabla u) \, dx,
	\end{split}
\end{equation}
where $\Omega_2':=\{x \in \Omega \ | \ u^p \geq m\}$ and it holds that  $$\displaystyle \int_{\Omega_2'} H^p(\nabla u) \, dx \rightarrow 0\qquad \text{for } m \rightarrow + \infty.$$
 Hence, passing to the limit for $m \rightarrow + \infty$ in \eqref{eq:computationsComparison3} we obtain that
\begin{equation}\label{eq:computationsComparison4}
	\begin{split}
		\int_{\{u^p-v^p \geq 0\}} (u^p+v^p)H^p(\nabla (\ln u -\ln v)) \, dx =0,
	\end{split}
\end{equation}
which clearly implies that
$$ u = K v \qquad \text{in the set } \{x \in \Omega \ | \ u(x) \geq v(x)\},$$
for some positive constant $K$. By our assumptions $\inf_{x \in \Omega} v > 0$ and $ u \leq v$ on $\partial \Omega$, hence it follows that $K=1$. But this implies that
$$ u \leq v \qquad \text{in } \Omega$$
and this complete the proof of this result in the case $p \geq 2$. The case $1<p<2$ follows repeating verbatim the proof of the case $p\ge 2$, but applying inequality \eqref{eq:comparisonSingular} instead of \eqref{eq:compaeisonDegenerate}.
\	
\end{proof}

Now we want to prove the corresponding result of Proposition \ref{prop:Comparison2} in exterior domains.

\begin{prop}\label{prop:Comparison3}
Let $\Omega$ be an exterior domain such that $\R^N \setminus \Omega$ is bounded and $f\in L^{\frac{N}{p}}(\Omega)$. Let $u \in \mathcal D^{1,p}(\Omega) $ be a weak positive subsolutions to \eqref{eq:comparison} and $v\in \mathcal D^{1,p}(\Omega)$ be a positive supersolution of

\begin{equation}\label{eq:comparison3}
	-\Delta_p^H v - \frac{\gamma}{H^\circ(x)^p}v^{p-1}=g(x) v^{p-1} \qquad \text{in } \Omega,
\end{equation} 
with $g\in L^{\frac{N}{p}}(\Omega)$.
Assume that $\inf_\Omega v >0$ and $f\le g$ in $\Omega$.
If $u \leq v$ on $\partial \Omega$ and
\begin{equation}\label{logAss}
	\limsup_{R \rightarrow + \infty} \frac{1}{R} \int_{{B}_{2R} \setminus {B}_R} u^p |\nabla \log v|^{p-1}=0,
\end{equation}
then $$u \leq v \quad \text{in } \Omega.$$ 
\end{prop}

\begin{proof}
In the same spirit of Proposition \ref{prop:Comparison2} we prove our result in the case $p \geq 2$. The other case is similar and it can be shown using the same arguments. To this aim, let $\varphi_R \in C^\infty_c({B}_{2R})$ be a standard cut-off function such that 
$$ \begin{cases}
	0 \leq \varphi_R \leq 1\\
	\varphi_R \equiv 1 \qquad \ \  \text{on } {B}_R\\
	|\nabla \varphi_R| \leq \frac{2}{R} \quad \text{on } {B}_{2R} \setminus {B}_R.
\end{cases}$$
Let us define
$$ \eta_1:=\varphi_R\frac{\min\{(u^p-v^p)^+,m\}}{u^{p-1}} \qquad \text{and} \qquad \eta_2:=\varphi_R \frac{\min\{(u^p-v^p)^+,m\}}{v^{p-1}}.$$
where $m>1$. As pointed out in the proof of previous proposition, it is possible to show, by standard arguments, that $\eta_1$ and $\eta_2$ are good test functions for the weak formulations \eqref{eq:comparison} and \eqref{eq:comparison3}. Hence, we obtain
\begin{equation}\label{eq:weakExterior1}
	\begin{split}
		\int_\Omega H^{p-1}(\nabla u) \langle \nabla H (\nabla u), \nabla \eta_1 \rangle \, dx &- \int_\Omega H^{p-1}(\nabla v) \langle \nabla H (\nabla v), \nabla \eta_2 \rangle \, dx\\
		&\le  \int_\Omega \frac{\gamma}{H^\circ(x)^p} \left(u^{p-1}\eta_1-v^{p-1}\eta_2\right) \, dx\\
		&+ \int_\Omega f(x) u^{p-1} \eta_1 -g(x)v^{p-1} \eta_2 \,dx \le 0,
	\end{split}
\end{equation}
since $f\le g$ in $\Omega.$
Now we explicitly compute the left hand side of \eqref{eq:weakExterior1}.
\begin{equation}\label{eq:weakExterior2}
	\begin{split}
		0\ge &\int_\Omega H^{p-1}(\nabla u) \langle \nabla H (\nabla u), \nabla \eta_1 \rangle \, dx - \int_\Omega H^{p-1}(\nabla v) \langle \nabla H (\nabla v), \nabla \eta_2 \rangle \, dx\\
		= &\int_{\Omega_1}  H^{p-1}(\nabla u) \langle \nabla H (\nabla u), \nabla \left(u-\frac{v^p}{u^p}u\right) \rangle \varphi_R   +  H^{p-1}(\nabla v) \langle \nabla H (\nabla v), \nabla \left(v-\frac{u^p}{v^p}v\right) \rangle \varphi_R \, dx\\
		&+\int_{\Omega_1} \left(u-\frac{v^p}{u^p}u\right) H^{p-1}(\nabla u) \langle \nabla H (\nabla u), \nabla \varphi_R \rangle  + \left(v-\frac{u^p}{v^p}v\right) H^{p-1}(\nabla v) \langle \nabla H (\nabla v), \nabla \varphi_R \rangle \, dx\\
		&+\int_{\Omega_2}  H^{p-1}(\nabla u) \langle \nabla H (\nabla u), \nabla \left(u^{1-p}\right)\rangle m \varphi_R \,dx +\int_{\Omega_2}  H^{p-1}(\nabla u) \langle \nabla H (\nabla u), \nabla \varphi_R\rangle mu^{1-p}  \,dx \\ &-\int_{\Omega_2}  H^{p-1}(\nabla v) \langle \nabla H (\nabla v), \nabla \left(v^{1-p}\right)\rangle m \varphi_R \,dx -  \int_{\Omega_2}  H^{p-1}(\nabla v) \langle \nabla H (\nabla v), \nabla \varphi_R \rangle mv^{1-p}\, dx\\
		=&: I_1+I_2+I_3+I_4,
	\end{split}
\end{equation}
where $\Omega_1 := \{x \in \Omega \ | \ 0 \leq u^p-v^p \leq m\}$ and $\Omega_2:= \{x \in \Omega \ | \ u^p-v^p \geq m\}$.
By Proposition \ref{prop:Comparison1} and using the definition of $\varphi_R$, it follows that there exits a positive constant depending only on $p$ such that
\begin{equation}\label{I_1}
I_1 \geq C_p \int_{\Omega_1 \cap {B}_R}(u^p+v^p) H^p(\nabla (\ln u -\ln v)) \, dx.
\end{equation}
Now we are going to give estimates for $I_2, I_3$ and $I_4$. We start with $I_2$. Using \eqref{eq:BddH} and the Cauchy-Schwarz inequality, setting $\tilde \Omega_1:=\{x\in \Omega: v^p\le u^p\},$ we have
\begin{equation}\label{I_2}
\begin{split}
|I_2| \leq& M\int_{\Omega_1} \left|u-\frac{v^p}{u^p}u\right| H^{p-1}(\nabla u) |\nabla \varphi_R| \,dx + M\int_{\Omega_1} \left|v-\frac{u^p}{v^p}v\right| H^{p-1} (\nabla v) |\nabla \varphi_R|\,dx\\
\le &  M\int_{\tilde \Omega_1} \left(2u H^{p-1}(\nabla u) + v H^{p-1} (\nabla v) \right) |\nabla \varphi_R| \,dx + M\int_{\tilde \Omega_1} u^p H^{p-1}(\nabla \ln v) |\nabla \varphi_R| \,dx \\
\leq & \frac{\mathcal{C}}{R}\int_{{B}_{2R} \setminus {B}_R}  u \cdot H^{p-1}(\nabla u) \,dx + \frac{\mathcal{C}}{R}\int_{{B}_{2R} \setminus {B}_R}  v \cdot H^{p-1} (\nabla v) \,dx\\
& + \frac{\mathcal{C}}{R} \int_{{B}_{2R} \setminus {B}_R}  u^p H^{p-1}(\nabla \ln v) \,dx \\
\leq & \mathcal{C} \left(\int_{{B}_{2R} \setminus {B}_R}  H^{p}(\nabla u) \,dx \right)^\frac{p-1}{p}\left(\int_{{B}_{2R} \setminus {B}_R}  u^{p^*} \,dx \right)^\frac{1}{p^*}\\
& + \mathcal{C} \left(\int_{{B}_{2R} \setminus {B}_R}  H^{p}(\nabla v) \,dx \right)^\frac{p-1}{p}\left(\int_{{B}_{2R} \setminus {B}_R}  v^{p^*} \,dx \right)^\frac{1}{p^*}\\
& +  \frac{\mathcal{C}}{R} \int_{{B}_{2R} \setminus {B}_R}  u^p H^{p-1}(\nabla \ln v) \,dx,
\end{split}
\end{equation}
where in the last line we applied the H\"older inequality with conjugate exponent $\left(N,p/(p-1),p^*\right)$ and $\mathcal{C}:=2M$. Passing to the limit for $R$ that goes to $+\infty$ in the right hand side of \eqref{I_2}, using also assumption \eqref{logAss} and \eqref{eq:equivalentNorm},  we deduce that $I_2$ goes to zero.

Now we proceed with the estimate of the term $I_3$. By \eqref{eq:EulerThm}, we have
\begin{equation}\nonumber
I_3 = m \int_{\Omega_2} H^{p-1}(\nabla u) \langle \nabla H (\nabla u), \nabla \varphi_R \rangle  u^{1-p} \,dx+ m(1-p) \int_{\Omega_2} H^p(\nabla u) \varphi_R u^{-p} \, dx
\end{equation}
Therefore
\begin{equation}\nonumber
	|I_3| \leq m M \int_{\tilde \Omega_2} H^{p-1}(\nabla u) |\nabla \varphi_R| u^{1-p} \,dx+ m(p-1) \int_{\tilde \Omega_2} H^p(\nabla u) \varphi_R u^{-p} \, dx,
\end{equation}
where $\tilde \Omega_2 := \{x \in \Omega \ | \  u^p \geq m\}$. Using this definition and also the properties of $\varphi_R$ we deduce that
\begin{equation}\label{I_34}
\begin{split}
	|I_3| &\leq \frac{2M}{R} \int_{\tilde \Omega_2}  u \cdot H^{p-1}(\nabla u)  \,dx+ (p-1) \int_{\tilde \Omega_2} H^p(\nabla u)  \, dx \\
	&\leq \mathcal{C} \left(\int_{{B}_{2R} \setminus {B}_R}  H^{p}(\nabla u) \,dx\right)^\frac{p-1}{p} \left(\int_{{B}_{2R} \setminus {B}_R}  u^{p^*} \,dx \right)^\frac{1}{p^*} + (p-1) \int_{\tilde \Omega_2} H^p(\nabla u)  \, dx,
\end{split}
\end{equation}
where $\mathcal{C}:=2M$. Passing to the limit for $m,R$ that go to $+\infty$, we deduce that 
\begin{equation}\label{I_3}
    \lim _{m,R\rightarrow + \infty} I_3=0.
\end{equation}

For the last term $I_4$, by \eqref{eq:EulerThm}, recalling $\tilde\Omega_2:= \{x\in \Omega: u^p\ge m\}$, we have
\begin{equation}\label{I_4}
	\begin{split}
		I_4 &= m(p-1) \int_{\Omega_2}  H^{p}(\nabla v) \varphi_R v^{-p}  \, dx -  \int_{\Omega_2}  mH^{p-1}(\nabla v) \langle \nabla H(\nabla v), \nabla \varphi_R \rangle v^{1-p}  \, dx\\
		&\geq -  M \int_{\Omega_2} m H^{p-1}(\nabla \ln v) | \nabla \varphi_R |\, dx \ge   M \int_{\tilde \Omega_2} u^p H^{p-1}(\nabla \ln v) | \nabla \varphi_R |\, dx\\
		&\geq -\frac{2M}{R} \int_{{B}_{2R} \setminus {B}_R} u^p H^{p-1}(\nabla 	\ln v)  \, dx
	\end{split}
\end{equation}
Hence, passing to the limit the right hand side of \eqref{I_4}, by \eqref{logAss} we have that $I_4$ goes to zero when $R$ tends to $+\infty$.

Finally, if we combine all the estimates \eqref{I_1}, \eqref{I_2}, \eqref{I_3}, \eqref{I_4} and we pass to the limit for $m, R \rightarrow + \infty$ we deduce that
\[\begin{split}
	0\ge&\int_\Omega H^{p-1}(\nabla u) \langle \nabla H (\nabla u), \nabla \eta_1 \rangle \, dx - \int_\Omega H^{p-1}(\nabla v) \langle \nabla H (\nabla v), \nabla \eta_2 \rangle \, dx\\
	=& \limsup_{R \rightarrow + \infty} (I_1 + I_2 + I_3 + I_4) \geq \liminf_{R \rightarrow + \infty} (I_1 + I_2 + I_3 + I_4) \\
	\geq &\int_{\{u \geq v\}} (u^p+v^p) H^p(\nabla(\ln u - \ln v)) \, dx \geq 0,
\end{split}\]
which implies $u \leq v$ in $\Omega$ as we concluded in the proof of Proposition \ref{prop:Comparison2}.

\end{proof}

\section{Proof of the main results} \label{AsympEst}

This section is dedicated to the proof of our main results:  Theorem \ref{thm:asymptEst}, Theorem  \ref{introtecnicanew} and Theorem \ref{stimagradiente}.

\begin{proof}[Proof of Theorem \ref{thm:asymptEst}]
We start by proving \eqref{eq:estat0}. To this aim, let us consider $u$ a solution of 
\begin{equation}\label{11}
    -\Delta_p^H u - \frac{\gamma}{[H^\circ(x)]^p} u^{p-1} = u^{p^*-p}u^{p-1} \qquad \text{in } \R^N.
    \end{equation}
    In particular, we have that $u$ is a subsolution of \eqref{11} in any bounded domain $\Omega=\mathcal{B}_{R_1}^{H^\circ}$. 
    We note that $f(x):=u^{p^*-p}\in L^{\frac{N}{p}}(\mathcal{B}_{R_1}^{H^\circ})$ and satisfies $|f(x)|\le A[H^{\circ}(x)]^{-\alpha}$ with $\alpha=\left((N-p)/{p}-\sigma_1\right)(p^*-p)<p$ for $x\in \mathcal{B}_{R_1}^{H^\circ}$ due to Theorem \ref{teoo}.
    By Proposition \ref{prop:SubSuperComp} we have that the function 	
$$v(H^\circ(x))=\frac{1-\delta [H^\circ(x)]^\varepsilon}{[H^\circ(x)]^{\mu_1}} \in \mathcal{D}^{1,p}(\mathcal{B}^{H^\circ}_{R_1})$$ 
is a positive supersolution of \eqref{eq:comparison2} in $\mathcal{B}^{H^\circ}_{R_1} \subset \Omega$, with $g\in L^{\frac{N}{p}}(\mathcal{B}^{H^\circ}_{R_1})$ satisfying $g(x)\ge A [H^{\circ}(x)]^{-\alpha}$ and where $0<\delta, \varepsilon, R_1<1$ are positive constants depending only on $N, p, \gamma, A$ and $\alpha.$

Let us consider $\Gamma >0$, $\mathcal{M}= \sup_{\partial \mathcal{B}_{R_1}^{H^\circ}} u$, $\mathcal{N}= \sup_{\partial \mathcal{B}_{R_1}^{H^\circ}} 1/v$ and define
$$w(x):= \mathcal{N}(\mathcal{M}+\Gamma)v(H^{\circ}(x)).$$
It is easy to chek that $w$ is a positive supersolution of \eqref{eq:comparison2} in $\mathcal{B}^{H^\circ}_{R_1} \subset \Omega$ and $\inf_{\mathcal{B}^{H^\circ}_{R_1}} w = \mathcal{M}+\Gamma>0$ and $u \leq w$ on $\partial \mathcal{B}^{H^\circ}_{R_1}$. Hence, by Proposition \ref{prop:Comparison2} we deduce that
$$u \leq w \qquad \text{in } \mathcal{B}^{H^\circ}_{R_1}.$$
Passing to the limit for $\Gamma \rightarrow 0$ we obtain that
$$u(x) \leq \frac{C}{[H^\circ(x)]^{\mu_1}} \qquad \text{in } \mathcal{B}^{H^\circ}_{R_1},$$
where $C=\mathcal{M} \cdot \mathcal{N}$.

Now we have to show the estimate from below. Let $u$ be a weak solution of \eqref{eq:ProbDoublyCritical}, then $u$ is a supersolution of 
\begin{equation}\label{111}
    -\Delta_p^H u - \frac{\gamma}{[H^\circ(x)]^p} u^{p-1} = 0 \qquad \text{in } \mathcal{B}^{H^\circ}_{R_1}.
    \end{equation}

We set $c_1:=\inf _{ \mathcal{B}_{R_1}^{H^\circ}} u>0$.

Now, we define
$$\tilde w (x) := c\frac{ c_1}{[H^\circ(x)]^{\mu_1}} \qquad \text{in } \mathcal{B}^{H^\circ}_{R_1},$$
where $c= \inf_{\partial \mathcal{B}^{H^\circ}_{R_1}} [H^\circ(x)]^{\mu_1}=R_1^{\mu_1}$ . The function $\tilde w$ is a subsolution to  \eqref{111} in $\mathcal{B}^{H^\circ}_{R_1}$. Since, it holds that $u \geq \tilde w$ in $\partial \mathcal{B}^{H^\circ}_{R_1}$, we conclude by using Proposition \ref{prop:Comparison2} to obtain
$$u \geq \tilde w \qquad \text{in } \mathcal{B}^{H^\circ}_{R_1},$$
and hence combining the estimates from above and below we deduce that \eqref{eq:estat0} is proved.

Now, our aim is to prove \eqref{eq:estatInf}. Let us consider $u$ a subsolution of 
\begin{equation}\label{22}
    -\Delta_p^H u - \frac{\gamma}{[H^\circ(x)]^p} u^{p-1} = u^{p^*-p}u^{p-1} \qquad \text{in } (\mathcal{B}_{R_2}^{H^\circ})^c.
    \end{equation}
    
    We note that $f(x):=u^{p^*-p}\in L^{\frac{N}{p}}((\mathcal{B}_{R_2}^{H^\circ})^c)$ and by Theorem \ref{teoo} we have $|f(x)|\le A[H^{\circ}(x)]^{-\alpha}$ with $\alpha=\left((N-p)/{p}+\sigma_2\right)(p^*-p)>p$ for $x\in (\mathcal{B}_{R_2}^{H^\circ})^c$.
    By Proposition \ref{prop:SubSuperComp}  the function 	
$$v(H^\circ(x))=\frac{1-\delta [H^\circ(x)]^{-\varepsilon}}{[H^\circ(x)]^{\mu_2}} \in \mathcal{D}^{1,p}((\mathcal{B}^{H^\circ}_{R_2})^c)$$ 
is a positive supersolution of \eqref{eq:comparison3} in $(\mathcal{B}^{H^\circ}_{R_2})^c \subset \Omega$, with $g\in L^{\frac{N}{p}}((\mathcal{B}^{H^\circ}_{R_2})^c)$ satisfying $g(x)\ge A [H^{\circ}(x)]^{-\alpha}$ and where $0<\delta, \varepsilon<1$ and $R_2>1$ are positive constants depending only on $N, p, \gamma, A$ and $\alpha.$

Let us consider $\Gamma >0$, $\mathcal{M}= \sup_{(\partial \mathcal{B}_{R_2}^{H^\circ})^c} u$, $\mathcal{N}= \sup_{(\partial \mathcal{B}_{R_2}^{H^\circ})^c} 1/v$ and define
$$w(x):= \mathcal{N}(\mathcal{M}+\Gamma)v(H^{\circ}(x)).$$
We note that $w$ is a positive supersolution of \eqref{eq:comparison3} in $(\mathcal{B}^{H^\circ}_{R_2})^c \subset \Omega$ and $\inf_{(\mathcal{B}^{H^\circ}_{R_2})^c} w = \mathcal{M}+\Gamma>0$ and $u \leq w$ on $(\partial \mathcal{B}^{H^\circ}_{R_2})^c$. 
We verify the condition \eqref{logAss}. Since $|\nabla \log v(x)|\le C|x|^{-1}$, by H\"older inequality we have 

\begin{equation}
    \begin{split}
      	\limsup_{R \rightarrow + \infty} \frac{1}{R} \int_{{B}_{2R} \setminus {B}_R} u^p |\nabla \log v|^{p-1}\,dx\le \limsup_{R \rightarrow + \infty}\frac{C}{R}\left(\int_{B_{2R}\setminus B_R}|u|^{p^*}\,dx\right)^{\frac{p}{p^*}}=0,  
    \end{split}
\end{equation}
where $C$ is a constant independent of $R$.

Hence, by Proposition \ref{prop:Comparison3} we deduce that
$$u \leq w \qquad \text{in } (\mathcal{B}^{H^\circ}_{R_2})^c.$$
Passing to the limit for $\Gamma \rightarrow 0$ we obtain that
$$u(x) \leq \frac{C}{[H^\circ(x)]^{\mu_2}} \qquad \text{in } (\mathcal{B}^{H^\circ}_{R_2})^c,$$
where $C=\mathcal{M} \cdot \mathcal{N}$

We conclude with the estimate from below.  Let $u$ be a weak solution of \eqref{eq:ProbDoublyCritical}.   Then $u$ is a supersolution of 
\begin{equation}\label{113}
    -\Delta_p^H u - \frac{\gamma}{[H^\circ(x)]^p} u^{p-1} = 0 \qquad \text{in } \R^N.
    \end{equation}
We claim that 
\begin{equation}\label{supersolution}
\int_{B_{2R}\setminus B_R}|\nabla \log u|^p \le CR^{N-p}
\end{equation}
for $R$ sufficiently large and constant $C$ independent of $R$. Indeed, by \eqref{113}, we have \begin{equation}\label{zio}
-\Delta_p^H u\ge 0\quad \text{in } \R^N.
\end{equation}
Therefore, considering the test function $\eta =\zeta ^p u^{1-p}$, with $\zeta \in C^1_c(\R^N)$ nonnegative function, and taking in \eqref{zio} we have 

\begin{equation}\label{zioo}
    \begin{split}
        -(p-1)\int_{\R^N} H^p(\nabla u) \zeta ^pu^{-p}\,dx+p\int_{\R^N} H^{p-1}(\nabla u)\langle \nabla H(\nabla u),\nabla \zeta\rangle \zeta^{p-1} u^{1-p}\,dx\ge 0.
    \end{split}
\end{equation}
By H\"older inequality and by $0$-homogeneity of $\nabla H$ we get 
\begin{equation}\label{aa}
  \int_{\R^N} H^p(\nabla \log u) \zeta ^p\,dx  \le M\frac{p}{p-1}\left(\int_{\R^N}H^p(\nabla \log u) \zeta ^p\,dx \right)^{\frac{p-1}{p}}\left(\int_{\R^N}|\nabla \zeta |^p\,dx\right)^{\frac{1}{p}}.
\end{equation}
Taking a standard cutoff function $\zeta$ in \eqref{aa} we get the claim \eqref{supersolution}. 

Now we set $c_2=\inf_{(\partial\mathcal{B}^{H^\circ}_{R_2})^c} u>0$ and  $c=\inf_{(\partial \mathcal{B}^{H^\circ}_{R_2})^c}[H^{\circ}(x)]^{\mu_2}$. We note that $v={c_2}c{[H^{\circ}(x)]^{-\mu_2}}$ is a weak solution of \eqref{113}. Moreover the condition \eqref{logAss} is verified. Indeed by H\"older inequality and \eqref{supersolution} we have 
\begin{equation}\label{oo}
\begin{split}
    \frac{1}{R} \int_{B_{2R}\setminus B_R} v^p|\nabla \log u|^{p-1}\,dx\le C R^{-1-p\mu_2+\frac{N}{p}}\left(\int_{B_{2R}\setminus B_R}|\nabla \log u|^p\,dx\right)^{\frac{p-1}{p}} \\ \le C R^{-1-p\mu_2+\frac{p-1}{p}(N-p)+\frac{N}{p}} \rightarrow 0
\end{split}
\end{equation}

since $\mu_2> (N-p)/{p}.$
Applying the Proposition \ref{prop:Comparison3} we conclude that $$u(x)\ge v(x)=c\frac{c_2}{[H^{\circ}(x)]^{\mu_2}}\quad \text{in } (\mathcal{B}^{H^\circ}_{R_2})^c$$ and therefore the thesis.
\end{proof}

Now we prove Theorem \ref{introtecnicanew} that will be essential to prove the asymptotic behavior of the gradient of solutions to \eqref{eq:ProbDoublyCritical}. For the reader convenience we state a more detailed statement contained in the following:

\begin{thm}\label{tecnicanew}
	Let $v\in C^{1,\alpha}_{loc}(\R^N\setminus \{0\})$ be a positive weak solution of the equation
	\begin{equation}\label{zioa}
		-\Delta_p^H v - \frac{\gamma}{[H^\circ(x)]^p} v^{p-1} = 0 \qquad \text{in } \R^N\setminus \{0\},
	\end{equation}
	where $0 \leq \gamma < C_H$. Assume that there exist two positive constants $C$ and $c$ such that 
	\begin{equation}\label{a_1}
		\frac{c}{[H^{\circ}(x)]^{\mu_2}}\le v(x)\le \frac{C}{[H^{\circ}(x)]^{\mu_2}} \qquad \forall x\in \R^N\setminus \{0\},
	\end{equation}
	and suppose that there exists a positive constant $\hat C$ such that 
	\begin{equation}\label{assunzionegradiente}
		|\nabla v(x)|\le \frac{\hat C}{[H^{\circ}(x)]^{\mu_2+1}} \qquad \ \  \forall x\in \R^N\setminus \{0\}, 
	\end{equation}
	then 
	\begin{equation}\label{v_1}
		v(x)=\frac{\overline c_1}{[H^{\circ}(x)]^{\mu_2}}, \qquad \qquad \qquad \qquad 
	\end{equation}
with 
	$$\overline{c_1}:=\limsup_{|x|\rightarrow 0} [H^{\circ}(x)]^{\mu_2}v(x). \qquad \ \ \ $$
	On the other hand, suppose that there exist two positive constants $\tilde C$ and $\tilde c$ such that 
	\begin{equation}\label{a_2}
		\frac{\tilde c}{[H^{\circ}(x)]^{\mu_1}}\le v(x)\le \frac{\tilde C}{[H^{\circ}(x)]^{\mu_1}} \qquad \forall x\in \R^N\setminus \{0\},
	\end{equation}
	and suppose that there exists a positive constant $\overline C$ such that 
	\begin{equation}\label{assunzionegradiente2}
		|\nabla v(x)|\le \frac{\overline C}{[H^{\circ}(x)]^{\mu_1+1}} \qquad \ \ \forall x\in \R^N\setminus \{0\}, 
	\end{equation}
	then  
	\begin{equation}\label{v_2}
		v(x)=\frac{\overline c_2}{[H^{\circ}(x)]^{\mu_1}}, \qquad \qquad \qquad \qquad 
	\end{equation}
	with $$\overline{c_2}:=\limsup_{|x|\rightarrow +\infty} [H^{\circ}(x)]^{\mu_1}v(x). \qquad \ \ \ $$
\end{thm}

\begin{proof}
   We prove this result only for $p\ge 2.$ The other case is similar. Suppose that \eqref{a_1} holds. From definition of $\overline c_1$, there exist $r$ (sufficiently small) such that 
   \begin{equation}\label{b_2}
       v(x)\leq \frac{\overline c_1+a_n}{[H^{\circ}(x)]^{\mu_2}}\quad \text{in } B_r(0)\setminus \{0\},
    \end{equation}
given $a_n\rightarrow 0$.

On the other hand, there exist sequences of radii $R_n$ and points $x_n$ with $R_n$ tending to zero and $H^{\circ}(x_n)=R_n$, such that 
    \begin{equation}\label{b_1}
       v(x_n)\ge \frac{\overline c_1-a_n}{[H^{\circ}(x_n)]^{\mu_2}}.
    \end{equation}
Now we set 
\begin{equation}\label{zeta1}
    w_n(x):= R_n^{\mu_2}v(R_nx) \qquad \forall x\in \mathcal{B}_{2}^{H^\circ}\setminus \mathcal{B}_{1/2}^{H^\circ}.
\end{equation}
 By \eqref{a_1}, it follows that $w_n$ is uniformly bounded in $L^{\infty}(\mathcal{B}_{2}^{H^\circ}\setminus \mathcal{B}_{1/2}^{H^\circ})$ and, since $w_n$ satisfies the equation \eqref{zioa}, by \cite{ACF,L,L2}, it is also uniformly bounded in $C^{1,\alpha}(K),$ for $0<\alpha<1$ and for any compact set $K\subset \mathcal{B}_{2}^{H^\circ}\setminus \mathcal{B}_{1/2}^{H^\circ}$. Moreover, from Ascoli-Arzela's Theorem, we deduce that $w_n\rightarrow w_{\infty}$ in the norm $\| \cdot\|_{C^{1,\alpha}(K)}$, for any compact set $K\subset \mathcal{B}_{2}^{H^\circ}\setminus \mathcal{B}_{1/2}^{H^\circ}.$
By \eqref{b_2}, we have $$w_{\infty}\le \frac{\overline c_1}{[H^{\circ}(x)]^{\mu_2}}$$ and by \eqref{b_1}, there exist a point $\overline x\in \partial \mathcal{B}_{1}^{H^\circ}$ such that $w_{\infty}(\overline x)=\overline c_1$. By the strong comparison principle \cite{lucio}, that holds under our assumption on $H$ (see Section~\ref{notations}), we have 
\begin{equation}\label{abcd}
w_{\infty}\equiv \frac{\overline c_1}{[H^{\circ}(x)]^{\mu_2}}  \quad \text{in }\mathcal{B}_{2}^{H^\circ}\setminus \mathcal{B}_{1/2}^{H^\circ}.
\end{equation}
Now we set 
\begin{equation}\label{eq:manmont}\hat v:=\frac{\overline c_1}{[H^{\circ}(x)]^{\mu_2}}
\end{equation} and we note that it solves \eqref{zioa}. Fix $R>0$ sufficiently large and let $\varphi_R \in C^\infty_c({B}_{2R})$ be a standard cut-off function such that 
\begin{equation}\label{tao} \begin{cases}
	0 \leq \varphi_R \leq 1\\
	\varphi_R \equiv 1 \qquad \ \  \text{on } {B}_R\\
	|\nabla \varphi_R| \leq \frac{2}{R} \quad \text{on } {B}_{2R} \setminus {B}_R.
\end{cases}
\end{equation}
Let us define
\begin{equation}\label{fun}\varphi_1:=\varphi_R\frac{\min\{(v^p-\hat v^p)^+,m\}}{v^{p-1}} \qquad \text{and} \qquad \varphi_2:=\varphi_R \frac{\min\{(v^p-\hat v^p)^+,m\}}{\hat v^{p-1}},\end{equation}
where $m>1$.

We remark that $\varphi_1$ and $\varphi_2$ are good test function in any domain $0\not \in \overline \Omega.$ Testing \eqref{fun} in \eqref{zioa} on the domain $\R^N\setminus \mathcal{B}_{R_n}^{H^\circ}$, and, since the stress field $H(\nabla u)^{p-1}\nabla H(\nabla u)\in W^{1,2}_{loc}(\R^N\setminus \mathcal{B}_{R_n}^{H^\circ})$ (see \cite{MSP}), exploiting the divergence theorem, we have 
\begin{equation}\label{sei}
 	\begin{split}
		& \int_{ \R^N\setminus \mathcal{B}_{R_n}^{H^\circ}} H^{p-1}(\nabla v) \langle \nabla H (\nabla v), \nabla \varphi_1 \rangle dx \,  - \int_{ \R^N\setminus\mathcal{B}_{R_n}^{H^\circ}} H^{p-1}(\nabla \hat v) \langle \nabla H (\nabla \hat v), \nabla \varphi_2 \rangle dx\, \\
		&- \int_{\R^N\setminus \mathcal{B}_{R_n}^{H^\circ}}\frac{\gamma}{H^\circ(x)^p} \left(v^{p-1}\varphi_1-v^{p-1}\varphi_2\right) \, dx\\
		&= \int_{\partial \mathcal{B}_{R_n}^{H^\circ}} H^{p-1}(\nabla v) \langle \nabla H (\nabla v), \eta_n \rangle \varphi_1\,  - H^{p-1}(\nabla \hat v) \langle \nabla H (\nabla \hat v), \eta_n \rangle\varphi_2 dx ,
	\end{split}   
\end{equation}
where $\eta_n$ is the inner unite normal vector at $\partial \mathcal{B}_{R_n}^{H^\circ}.$  
Now we set $\mathcal{A}:= \{ 0\le v^p-\hat v^p\le m\}$ and $\mathcal B:=\{ v^p-\hat v^p\ge m\}$. Then \eqref{sei} becomes: 
\begin{equation}\label{sette}
    \begin{split}
        &\int_{\partial \mathcal{B}_{R_n}^{H^\circ}} -H^{p-1}(\nabla v) \langle \nabla H (\nabla v), \eta_n \rangle \varphi_1\,  + H^{p-1}(\nabla \hat v) \langle \nabla H (\nabla \hat v), \eta_n \rangle\varphi_2 \,dx 
        \\ &=\int_{(\R^N\setminus \mathcal{B}_{R_n}^{H^\circ})\cap \mathcal A}  H^{p-1}(\nabla v) \langle \nabla H (\nabla v), \nabla \left(v-\frac{\hat v^p}{v^p}v\right) \rangle \varphi_R\,dx   \\&+\int_{(\R^N\setminus \mathcal{B}_{R_n}^{H^\circ})\cap \mathcal A}  H^{p-1}(\nabla \hat v) \langle \nabla H (\nabla \hat v), \nabla \left(\hat v-\frac{v^p}{\hat v^p}\hat v\right) \rangle \varphi_R \, dx\\
		&+\int_{(\R^N\setminus \mathcal{B}_{R_n}^{H^\circ})\cap \mathcal A} \left(v-\frac{\hat v^p}{v^p}v\right) H^{p-1}(\nabla v) \langle \nabla H (\nabla v), \nabla \varphi_R \rangle \,dx  \\ &+\int_{(\R^N\setminus \mathcal{B}_{R_n}^{H^\circ})\cap \mathcal A} \left(\hat v-\frac{v^p}{\hat v^p}\hat v\right) H^{p-1}(\nabla \hat v) \langle \nabla H (\nabla \hat v), \nabla \varphi_R \rangle \, dx\\
		&+\int_{(\R^N\setminus \mathcal{B}_{R_n}^{H^\circ})\cap \mathcal B}  H^{p-1}(\nabla v) \langle \nabla H (\nabla v), \nabla \left(v^{1-p}\right)\rangle m \varphi_R +  H^{p-1}(\nabla v) \langle \nabla H (\nabla v), \nabla \varphi_R\rangle mv^{1-p}  \,dx \\ &-\int_{(\R^N\setminus \mathcal{B}_{R_n}^{H^\circ})\cap \mathcal B}  \left(H^{p-1}(\nabla \hat v) \langle \nabla H (\nabla \hat v), \nabla \left(\hat v^{1-p}\right)\rangle m \varphi_R  +   H^{p-1}(\nabla \hat v) \langle \nabla H (\nabla \hat v), \nabla \varphi_R \rangle m\hat v^{1-p}\right)\, dx\\
		=&: I_1+I_2+I_3+I_4+I_5+I_6.
    \end{split}
\end{equation}
Now we set $\Omega_1:= (\R^N\setminus \mathcal{B}_{R_n}^{H^\circ})\cap \mathcal A$ and $\Omega_2:=(\R^N\setminus \mathcal{B}_{R_n}^{H^\circ})\cap \mathcal B$. 
By Proposition \ref{prop:Comparison1}, it follows that there exits a positive constant depending only on $p$ such that
\begin{equation}\label{I_1a}
I_1 +I_2\geq C_p \int_{\Omega_1\cap {B}_{2R}}(v^p+\hat v^p) H^p(\nabla (\ln v -\ln \hat v))\varphi_R \, dx.
\end{equation}
Now we estimate $I_3+I_4$. Using \eqref{eq:BddH}   and the H\"older inequality with conjugate exponent $\left(N,{p}/({p-1}),p^*\right)$ we have
\begin{equation}\label{I_2a}
\begin{split}
|I_3+I_4| \leq& M\int_{\Omega_1} \left|v-\frac{\hat v^p}{v^p}v\right| H^{p-1}(\nabla v) |\nabla \varphi_R| \,dx + M\int_{\Omega_1} \left|\hat v-\frac{v^p}{\hat v^p}\hat v\right| H^{p-1} (\nabla \hat v) |\nabla \varphi_R|\,dx\\
\le &  M\int_{\Omega_1} \left(2v H^{p-1}(\nabla v) + \hat v H^{p-1} (\nabla \hat v) \right) |\nabla \varphi_R| \,dx + M\int_{\Omega_1} v^p H^{p-1}(\nabla \ln \hat v) |\nabla \varphi_R| \,dx \\
\leq & \frac{2M}{R}\int_{{B}_{2R} \setminus {B}_R}  v  H^{p-1}(\nabla v) \,dx + \frac{2M}{R}\int_{{B}_{2R} \setminus {B}_R}  \hat v  H^{p-1} (\nabla \hat v) \,dx\\
& + \frac{2M}{R} \int_{{B}_{2R} \setminus {B}_R}  v^p H^{p-1}(\nabla \ln \hat v) \,dx \\
\leq & 2M \left(\int_{{B}_{2R} \setminus {B}_R}  H^{p}(\nabla v) \,dx \right)^\frac{p-1}{p}\left(\int_{{B}_{2R} \setminus {B}_R}  v^{p^*} \,dx \right)^\frac{1}{p^*}\\
& + 2M \left(\int_{{B}_{2R} \setminus {B}_R}  H^{p}(\nabla \hat v) \,dx \right)^\frac{p-1}{p}\left(\int_{{B}_{2R} \setminus {B}_R}  \hat v^{p^*} \,dx \right)^\frac{1}{p^*}\\
& +  \frac{2M}{R} \int_{{B}_{2R} \setminus {B}_R}  v^p H^{p-1}(\nabla \ln \hat v) \,dx.
\end{split}
\end{equation}
By \eqref{eq:manmont} we have $H(\nabla \ln \hat v(x))\le C[H^{\circ}(x)]^{-1}$, where $C=\mu_2$, we have that 
\begin{equation}\label{zizi}
\begin{split}
    \frac{1}{R}\int_{B_{2R}\setminus B_R} v^pH^{p-1}(\nabla \ln \hat v) &\le \frac{C}{R^p}\int_{B_{2R}\setminus B_R} |v|^p 
\\ &\le C\left(\int_{B_{2R}\setminus B_R}|v|^{p^*}\right)^{\frac{p}{p^*}},
    \end{split}
\end{equation}
where we used the H\"older inequality.
Since $v$ has the right summability at the infinity, for $R$ that goes to infinity in the right hand side of \eqref{I_2a}, we obtain that $I_3+I_4$ goes to zero.

Now we proceed with the estimate of the term $I_5$.
Recalling that $v^p \geq m$ in $\Omega_2$, we get 
\begin{equation}\nonumber
	|I_5| \leq m M \int_{\Omega_2} H^{p-1}(\nabla v) |\nabla \varphi_R| v^{1-p} \,dx+ m(p-1) \int_{ \Omega_2} H^p(\nabla v) \varphi_R v^{-p} \, dx,
\end{equation}
Using  the properties of $\varphi_R$ we obtain that
\begin{equation}\label{I_3a}
\begin{split}
	|I_5| &\leq \frac{M}{R} \int_{{B}_{2R} \setminus {B}_R}  v  H^{p-1}(\nabla v)  \,dx+ (p-1) \int_{ \Omega_2} H^p(\nabla v)  \, dx \\
	&\leq M \left(\int_{{B}_{2R} \setminus {B}_R}  H^{p}(\nabla v) \,dx\right)^\frac{p-1}{p} \left(\int_{{B}_{2R} \setminus {B}_R}  v^{p^*} \,dx \right)^\frac{1}{p^*} + (p-1) \int_{ \Omega_2} H^p(\nabla v)  \, dx,
\end{split}
\end{equation}
 Passing to the limit for $m,R$ that go to $+\infty$, we deduce that $I_5$ goes to zero since the set $ \Omega_2$ vanishes as $m \rightarrow +\infty$.

For the last term $I_6$ we obtain
\begin{equation}\label{I_4a}
	\begin{split}
		I_6 &= m(p-1) \int_{\Omega_2}  H^{p}(\nabla \hat v) \varphi_R \hat v^{-p}  \, dx -  \int_{\Omega_2}  mH^{p-1}(\nabla \hat v) \langle \nabla H(\nabla \hat v), \nabla \varphi_R \rangle \hat v^{1-p}  \, dx\\
		&\geq -  M \int_{\Omega_2} m H^{p-1}(\nabla \ln \hat v) | \nabla \varphi_R |\, dx\\
		&\geq -\frac{2M}{R} \int_{{B}_{2R} \setminus {B}_R} v^p H^{p-1}(\nabla 	\ln \hat v)  \, dx
	\end{split}
\end{equation}
Hence, passing to the limit in the right hand side of \eqref{I_4a} and using \eqref{zizi} we have that the right hand of \eqref{I_4a} tends to zero when $R$ tends to the infinity. 

Now we estimate the left hand of \eqref{sette}. By \eqref{fun} and Proposition \ref{prop:Lindqvist} (in particular see \eqref{intro:eq:ineqStandardBis}), we get
\begin{equation}
    \begin{split}
        &\int_{\partial \mathcal{B}_{R_n}^{H^\circ}} -H^{p-1}(\nabla \hat v) \langle \nabla H (\nabla \hat v), \eta_n \rangle \varphi_2\,  + H^{p-1}(\nabla  v) \langle \nabla H (\nabla  v), \eta_n \rangle\varphi_1 \,dx \\ & =  -\int_{\partial \mathcal{B}_{R_n}^{H^\circ}} \varphi_1 \left(H^{p-1}(\nabla \hat v) \langle \nabla H (\nabla \hat v), \eta_n \rangle \,  - H^{p-1}(\nabla  v) \langle \nabla H (\nabla  v), \eta_n \rangle\right) \,dx \\ &-\int_{\partial \mathcal{B}_{R_n}^{H^\circ}} (\varphi_2-\varphi_1 )H^{p-1}(\nabla \hat v) \langle \nabla H (\nabla \hat v), \eta_n \rangle \,dx \\& \le
        \int_{\partial \mathcal{B}_{R_n}^{H^\circ}}\frac{m}{v^{p-1}}\tilde C_p (|\nabla\hat v|^{p-2}+|\nabla v|^{p-2})|\nabla \hat v-\nabla v| +\frac{2m}{\hat v^{p-1}}M\alpha_2 |\nabla \hat v|^{p-1}\,dx, \\ &:=J_1+J_2.
    \end{split}
\end{equation}

where, in the last line, we used \eqref{eq:equivalentNorm} and \eqref{eq:BddH}.

We set $x=R_ny$, with $y\in \partial \mathcal{B}_{1}^{H^\circ}$. Using this change of variables and recalling \eqref{a_1}, \eqref{zeta1} and \eqref{abcd}, $J_1$ can be estimated as 

\begin{equation}\label{aaa1}
\begin{split}
    J_1\le \tilde C_1 &\int_{\partial \mathcal{B}_{1}^{H^\circ}} mR_n^{N-1+\mu_2(p-1)-(\mu_2+1)(p-1)}  |-\mu_2\overline c_1\nabla H^{\circ}(y)-\nabla w_n(y)|\,dy \\ &=\tilde C_1 R_n^{N-p-\epsilon}
    \end{split}
\end{equation}
where we choose $m=R_n^{-\epsilon}$, for $\epsilon>0$ fixed sufficiently small, and $\tilde C_1$ is a positive constant.

In a similar way $J_2$ is estimated as 
\begin{equation}\label{aaa}
   J_2\le \tilde C_2 R_n^{N-p-\epsilon},
\end{equation}
where $\tilde C_2$ is a positive constant.

Finally, if we combine all the estimates \eqref{I_1a}, \eqref{I_2a}, \eqref{I_3a}, \eqref{I_4a}, \eqref{aaa1}, \eqref{aaa} and passing to the limit for $R_n\rightarrow 0$, and then, exploiting the Fatou's lemma, for $R\rightarrow +\infty$,  we deduce that
\[\begin{split}
	&\int_{\{v \geq \hat v\}} (v^p+\hat v^p) H^p(\nabla(\ln v - \ln \hat v)) \, dx=0,
\end{split}\]
which implies $v \leq \hat v$ as we concluded in the proof of Proposition \ref{prop:Comparison2}. 

To prove that $\hat v\le v$, let us consider 
\begin{equation}\label{fun2}\varphi_1:=\varphi_R\frac{\min\{(\hat v^p-v^p)^+,m\}}{\hat v^{p-1}} \qquad \text{and} \qquad \varphi_2:=\varphi_R \frac{\min\{(\hat v^p- v^p)^+,m\}}{ v^{p-1}},\end{equation}
where $m>1$ and $\varphi_R$ is the standard cutoff function defined in \eqref{tao}. Using \eqref{fun2} in the weak formulation of \eqref{zioa}, proceeding in a similar way as above we obtain \eqref{v_1}. The only thing to check is that 
\begin{equation}\label{ap}
    \frac{1}{R}\int_{B_{2R}\setminus B_R} \hat v^pH^{p-1}(\nabla \ln  v)\,dx\rightarrow 0.
\end{equation}
Assumption \eqref{assunzionegradiente} ensures \eqref{ap}, and therefore we are done.

Now, assume that \eqref{a_2} holds and suppose that $p\ge 2$, the other case is similar. The proof is similar to the previous one and, for this reason, we omit some details. From definition of $\overline c_2$, there exist $r$ (sufficiently large) such that

\begin{equation}\label{b_112}
        v(x)\leq\frac{\overline c_2+a_n}{[H^{\circ}(x)]^{\mu_1}} \quad \text{in } \R^N\setminus B_r(0),
    \end{equation}
given $a_n\rightarrow 0$.

On the other hand, there exist sequences of radii $R_n$ and points $x_n$ with $R_n$ tending to infinity and $H^{\circ}(x_n)=R_n$, such that 
\begin{equation}\label{b_11}
        v(x_n)\geq\frac{\overline c_2-a_n}{[H^{\circ}(x_n)]^{\mu_1}}.
    \end{equation}
    Now we set 
\begin{equation}\label{zeta2}
    w_n(x):= R_n^{-\mu_1}v\left(\frac{x}{R_n}\right) \qquad \forall x\in \mathcal{B}_{2}^{H^\circ}\setminus \mathcal{B}_{1/2}^{H^\circ}.
\end{equation}
As in the previous case we obtain 
\begin{equation}\label{abcde}
w_{\infty}\equiv \frac{\overline c_2}{[H^{\circ}(x)]^{\mu_1}}  \quad \text{in }\mathcal{B}_{2}^{H^\circ}\setminus \mathcal{B}_{1/2}^{H^\circ}.
\end{equation}
Now we set $$\hat v:=\frac{\overline c_2}{[H^{\circ}(x)]^{\mu_1}}$$ and we remark that it solves \eqref{zioa}. We show that $v=\hat v$ in $\R^N$. 
To this aim, fix $\varepsilon>0$ sufficiently small and let $\varphi_\varepsilon \in C^\infty(\R^N)$ be a function such that 
\begin{equation}\label{tao2} \begin{cases}
	0 \leq \varphi_\varepsilon \leq 1\\
	\varphi_\varepsilon \equiv 0 \qquad \ \  \text{on } {B}_\varepsilon\\
	\varphi_\varepsilon \equiv 1 \qquad \ \  \text{on } ({B}_{2\varepsilon})^c\\
	|\nabla \varphi_\varepsilon| \leq \frac{2}{\varepsilon} \quad \text{on } {B}_{2\varepsilon} \setminus {B}_\varepsilon.
\end{cases}
\end{equation}

Let us define
\begin{equation}\label{fun22}\varphi_1:=\varphi_\varepsilon\frac{\min\{(v^p-\hat v^p)^+,m\}}{v^{p-1}} \qquad \text{and} \qquad \varphi_2:=\varphi_\varepsilon\frac{\min\{(v^p-\hat v^p)^+,m\}}{\hat v^{p-1}},\end{equation}
where $m>1$.

We note that $\varphi_1$ and $\varphi_2$ are good test function in any bounded domain $0\in \Omega .$ Testing \eqref{fun22} in \eqref{zioa} on the domain $\mathcal{B}_{R_n}^{H^\circ}$ we get 
\begin{equation}\label{six}
 	\begin{split}
		& \int_{ \mathcal{B}_{R_n}^{H^\circ}} H^{p-1}(\nabla v) \langle \nabla H (\nabla v), \nabla \varphi_1 \rangle dx \,  - \int_{ \mathcal{B}_{R_n}^{H^\circ}} H^{p-1}(\nabla \hat v) \langle \nabla H (\nabla \hat v), \nabla \varphi_2 \rangle dx\, \\
		&- \int_{ \mathcal{B}_{R_n}^{H^\circ}}\frac{\gamma}{H^\circ(x)^p} \left(v^{p-1}\varphi_1-v^{p-1}\varphi_2\right) \, dx\\
		&= \int_{\partial \mathcal{B}_{R_n}^{H^\circ}} H^{p-1}(\nabla v) \langle \nabla H (\nabla v), \eta_n \rangle \varphi_1\,  - H^{p-1}(\nabla \hat v) \langle \nabla H (\nabla \hat v), \eta_n \rangle\varphi_2 \,dx ,
	\end{split}   
\end{equation}
where $\eta_n$ is the outward unite normal vector at $\partial \mathcal B_{R_n}^{H^{\circ}}.$ By  \eqref{fun22} and Proposition \ref{prop:Lindqvist}, the right hand of \eqref{six}
 becomes 

\begin{equation}
    \begin{split}
        &\int_{\partial \mathcal{B}_{R_n}^{H^\circ}} -H^{p-1}(\nabla \hat v) \langle \nabla H (\nabla \hat v), \eta_n \rangle \varphi_2\,  + H^{p-1}(\nabla  v) \langle \nabla H (\nabla  v), \eta_n \rangle\varphi_1 \,dx \\ & =  \int_{\partial \mathcal{B}_{R_n}^{H^\circ}} -\varphi_1 \left(H^{p-1}(\nabla \hat v) \langle \nabla H (\nabla \hat v), \eta_n \rangle \,  - H^{p-1}(\nabla  v) \langle \nabla H (\nabla  v), \eta_n \rangle\right) \,dx \\ &-\int_{\partial \mathcal{B}_{R_n}^{H^\circ}} (\varphi_2-\varphi_1 )H^{p-1}(\nabla \hat v) \langle \nabla H (\nabla \hat v), \eta_n \rangle \,dx \\& \le
        \int_{\partial \mathcal{B}_{R_n}^{H^\circ}}\tilde C_pv (|\nabla\hat v|^{p-2}+|\nabla v|^{p-2})(\nabla \hat v-\nabla v) +\frac{v^p}{\hat v^{p-1}}M\alpha_2 |\nabla \hat v|^{p-1}\,dx, \\ &:=\tilde J_1+\tilde J_2.
    \end{split}
\end{equation}
where, in the last line, we used \eqref{eq:equivalentNorm} and \eqref{eq:BddH}.

We consider the following change of variables
$x=y/R_n$, with $y\in \partial \mathcal{B}_{1}^{H^\circ}$. Using this fact, by \eqref{zeta2} and \eqref{abcde}, we have  

\begin{equation}
    \begin{split}
        \tilde J_1 &\le  \overline C_1\int_{\partial \mathcal{B}_{1}^{H^\circ}}R_n^{1-N-\mu_1+(\mu_1+1)(p-1)}|-\mu_1\overline c_2y-\nabla w_n(y)|\,dy  \\ &=\overline C_1R_n^{-N+\mu_1p+p}
    \end{split}
\end{equation}

where $ \overline C_1$ is a positive constant.

In a similar way $\tilde J_2$ is estimated as 
\begin{equation}\label{aaa4}
   \tilde J_2\le  \overline C_2 R_n^{-N+\mu_1p+p},
\end{equation}
where $\overline C_2$ is a positive constant. We note that $-N+\mu_1p+p<0$ since $0\le \mu_1<(N-p)/p$. 
Proceeding as in the case \eqref{a_1}, we prove that 
\begin{equation}
\int_{\{v \geq \hat v\}} (v^p+\hat v^p) H^p(\nabla(\ln v - \ln \hat v)) \, dx=0,
\end{equation}
which implies $v \leq \hat v$ in $\Omega$ as we concluded in the proof of Proposition \ref{prop:Comparison2}. 

To prove that $\hat v\le v$, let us consider 
\begin{equation}\label{fun23}\varphi_1:=\varphi_\varepsilon\frac{\min\{(\hat v^p-v^p)^+,m\}}{\hat v^{p-1}} \qquad \text{and} \qquad \varphi_2:=\varphi_\varepsilon\frac{\min\{(\hat v^p- v^p)^+,m\}}{ v^{p-1}},\end{equation}
where $m>1$ and $\varphi_\varepsilon$ is defined as \eqref{tao2}. Testing \eqref{fun23} in \eqref{zioa}, using the assumption \eqref{assunzionegradiente2} and proceeding in a similar way as above we get \eqref{v_2}.
\end{proof}

At this point we are ready to prove 

\begin{proof}[Proof of Theorem \ref{stimagradiente}]
We prove \eqref{eq:estatInfgrad}, the other case is similar and it can be proved in a similar way. We start by showing the estimate from above. For $R_n$ tending to infinity, let us consider 
\begin{equation}\label{riscalata}
w_n(x):=R_n^{\mu_2}u(R_nx) \qquad \forall x\in  \mathcal{B}_{2}^{H^\circ}\setminus \mathcal{B}_{1/2}^{H^\circ}.  
\end{equation}
We remark that $w_n$ is uniformly bounded in $L^{\infty}(\mathcal{B}_{2}^{H^\circ}\setminus \mathcal{B}_{1/2}^{H^\circ})$ and it weakly solves 
\begin{equation}
    -\Delta_p^Hw_n- \frac{\gamma}{H^\circ(x)^p} w_n^{p-1}= R_n^{\mu_2(p-1)+p-\mu_2(p^*-1)}w_n^{p^*-1} \qquad \text{in } \R^N.
\end{equation}

Since $\mu_2(p-1)+p-\mu_2(p^*-1)<0$, by \cite{ACF,L,L2}, $w_n$ is also uniformly bounded in $C^{1,\alpha}(K),$ for $0<\alpha<1$ and for any compact set $K\subset \mathcal{B}_{2}^{H^\circ}\setminus \mathcal{B}_{1/2}^{H^\circ}$. For $R_n$ sufficiently large we get the estimate from above in \eqref{eq:estatInfgrad}.

Now we prove the estimate from below. Suppose by contradiction that there exist sequences of points $x_n$ such that 
\begin{equation}\label{bau}
[H^{\circ}(x_n)]^{\mu_2+1}|\nabla u(x_n)|\rightarrow 0 \qquad \text{for } |x_n|\rightarrow \infty.
\end{equation}
For $0<a<A$ fixed, let us consider 
$$w_n(x):=R_n^{\mu_2}u(R_nx).$$ For $n$ sufficiently large, from Theorem \ref{thm:asymptEst}, relabeling the constants, we have 
$$\frac{c}{A^{\mu_2}}\le w_n(x)\le \frac{C}{a^{\mu_2}}\qquad \text{in } \overline{\mathcal{B}_{A}^{H^\circ}\setminus \mathcal{B}_{a}^{H^\circ}}.$$ 
Furthermore, recalling the estimate from above of the gradients of the weak solution $u$, proved previously, we get
\begin{equation}
    |\nabla w_n(x)|\le \frac{\tilde C}{a^{\mu_2+1}} \quad\text{in } \overline{\mathcal{B}_{A}^{H^\circ}\setminus \mathcal{B}_{a}^{H^\circ}}.
\end{equation}
For $a,A$ fixed, by \cite{ACF,L,L2}, $w_n$ is also uniformly bounded in $C^{1,\alpha}(K),$ for $0<\alpha<1$ and for any compact set $K\subset \mathcal{B}_{A}^{H^\circ}\setminus \mathcal{B}_{a}^{H^\circ}$. Moreover $$w_n(x)\rightarrow w_{a,A} \qquad \text{in } B_A\setminus B_a,$$ in the norm $C^{1,\alpha '}$, for $0<\alpha' <\alpha.$ 
Moreover, since $w_n$ weakly solves
\begin{equation}-\Delta_p^Hw_n- \frac{\gamma}{H^\circ(x)^p} w_n^{p-1}= R_n^{\mu_2(p-1)+p-\mu_2(p^*-1)}w_n^{p^*-1} \qquad \text{in } \mathcal{B}_{A}^{H^\circ}\setminus \mathcal{B}_{a}^{H^\circ},
\end{equation}

 we deduce that 
\begin{equation}
-\Delta_p^Hw_{a,A}- \frac{\gamma}{H^\circ(x)^p} w_{a,A}^{p-1}=0 \qquad \text{in } \mathcal{B}_{A}^{H^\circ}\setminus \mathcal{B}_{a}^{H^\circ}.
\end{equation}
Now we take $a_j=1/j$ and $A_j=j$, for $j\in \N$ and we construct $w_{a_j,A_j}$ as above. For $j$ goes to infinity, using a standard diagonal process, we construct a limiting profile $w_{\infty}$ so that 
\begin{equation}
-\Delta_p^Hw_{\infty}- \frac{\gamma}{H^\circ(x)^p} w_{\infty}^{p-1}=0 \qquad \text{in } \R^N\setminus \{0\},
\end{equation}
with $w_{\infty}\equiv w_{a_j,A_j}$ in $\mathcal{B}_{A_j}^{H^\circ}\setminus \mathcal{B}_{a_j}^{H^\circ}$.

Since $w_{\infty}$ satisfies the assumptions \eqref{a_1}, \eqref{assunzionegradiente} of the Theorem \ref{tecnicanew}, we get 
\begin{equation}
    w_{\infty}(x)= \frac{\overline c_2}{[H^{\circ}(x)]^{\mu_2}},
\end{equation}
where we set $\overline{c_2}:=\limsup_{|x|\rightarrow 0} [H^{\circ}(x)]^{\mu_2}w_{\infty}(x)$.

Now we set $y_n=x_n/R_n$, and by \eqref{bau}, we deduce that $|\nabla w_n(y_n)|$ tends to zero as $R_n$ tends to infinity. This fact and the uniform convergence of the gradients imply that there exist $\overline y\in \partial \mathcal{B}_{1}^{H^\circ}$ such that
$$|\nabla w_{\infty}(\overline y)|=0.$$
This is an absurd since the solution $w_{\infty}$ has no critical points.
\end{proof}

\appendix
\section{}\label{appendix}
\label{existenceofSol}
In this section we prove the existence of a weak solution to problem \eqref{eq:ProbDoublyCritical} by means of a minimization problem. For this reason we  define the following minimization problem:
Let $S(\gamma)$ defined as
\begin{equation}\label{eq:quotient}
	S(\gamma):= \inf_{u\in \mathcal{D}^{1,p}(\R^N)\setminus \{0\}}\frac{\mathcal L(u)}{\displaystyle \left(\int_{\mathbb R^N}|u|^{p^*}\, dx\right) ^\frac{p}{p^*}},\end{equation}
where 
\begin{equation}\label{eq:lampedus}
	{\mathcal L(u)=\int_{\mathbb R^N}H^p(\nabla u)-\gamma \frac{|u|^p}{H^\circ(x)^p} \, dx.}
\end{equation}
Thanks to the Hardy inequality \eqref{eq:Hardy} we deduce that $S(\gamma)\geq0$.

\begin{thm}
    Let $0\le\gamma< C_H$. The problem \eqref{eq:ProbDoublyCritical} has a positive weak solution that it minimizes the quotient \eqref{eq:quotient}.
\end{thm}

\begin{proof} 
	Let $\{u_n\}$ be a minimizing sequence to \eqref{eq:quotient}. Without loss of generality, because of the homogeneity of the quotient in  
	\eqref{eq:quotient}, we assume that 
	\begin{equation}\label{eq:slavin}
		\mathcal L(u_n)=1.\end{equation}
	By Hardy inequality (observe that $H$ is a norm equivalent to the euclidean one), 
	$\mathcal L(\cdot)^{1/p}$
	is an equivalent norm to the standard   one $\|\cdot \|_{\mathcal{D}^{1,p}(\R^N)}$, we have that 
	\[
	\|u_n\|_{\mathcal{D}^{1,p}(\R^N)}\leq C,
	\]
	with $C$ that does not depend on $n$. Hence up to a subsequence $u_n \rightharpoonup u_0$ in $\mathcal{D}^{1,p}(\R^N)$. Moreover let us also assume  that (we will prove it later)
	\begin{equation}\label{eq:weaklimit}
		u_0 \not\equiv0.
	\end{equation}
	Recalling \eqref{eq:quotient}, using the weak lower semicontinuity of the norm and \eqref{eq:slavin}, we deduce
	\begin{equation}\label{eq:Slambda}
		\begin{split}
		S(\gamma)\leq \frac{\mathcal L(u_0)}{\left(\displaystyle\int_{\mathbb R^N}u_0^{p^*}\, dx\right) ^\frac{p}{p^*}}&\leq \liminf_n \frac{\mathcal L(u_n)}{\displaystyle \left(\int_{\mathbb R^N}u_0^{p^*}\, dx\right) ^\frac{p}{p^*}}\\
		& =\frac{1}{\displaystyle\left(\int_{\mathbb R^N}u_0^{p^*}\, dx\right) ^\frac{p}{p^*}}
		\end{split}
	\end{equation}
	Using Sobolev inequality we have that  $u_n\rightarrow u_0$ a.e. in $\mathbb R^N$. Therefore by the Breis-Lieb~result \cite{BL}, it follows that 
	\[\int_{\mathbb R^N}u_0^{p^*}\, dx=\int_{\mathbb R^N}u_n^{p^*}\, dx- \int_{\mathbb R^N}(u_n-u_0)^{p^*}\, dx+o(1).\]
	Then from \eqref{eq:Slambda} we obtain 
	\begin{equation}\label{eq:slevin1}
		S(\gamma)\leq \frac{1}{\displaystyle \left(\int_{\mathbb R^N}u_n^{p^*}\, dx- \int_{\mathbb R^N}(u_n-u_0)^{p^*}\, dx+o(1)\right) ^\frac{p}{p^*}}
	\end{equation}
	From \eqref{eq:quotient}, we deduce
	\[\mathcal L(u_n)=S(\gamma) \left(\int_{\mathbb R^N}u_n^{p^*}\, dx\right) ^\frac{p}{p^*}+o(1) \]
	and
	\[\mathcal L(u_0-u_n)\geq S(\gamma) \left(\int_{\mathbb R^N}|u_0-u_n|^{p^*}\, dx\right) ^\frac{p}{p^*}.\]
	Using these last inequalities in \eqref{eq:slevin1} we get
	\begin{equation}\label{eq:slevin2}
	\begin{split}
		S(\gamma)&\leq S(\gamma)\frac{1}{\left(\mathcal L^\frac{p^*}{p}(u_n)- \mathcal L^\frac{p^*}{p}(u_0-u_n)+o(1)\right) ^\frac{p}{p^*}}\\
		&\leq S(\gamma)\frac{1}{\left( 1- \mathcal L^\frac{p^*}{p}(u_0-u_n)+o(1)\right) ^\frac{p}{p^*}},
	\end{split}
	\end{equation}
	where in the last line we used \eqref{eq:slavin}. Taking the limit superior of \eqref{eq:slevin2} we get a contradicition, i.e. $S(\gamma)<S(\gamma)$ unless the sequence $\mathcal L(u_0-u_n)\rightarrow 0$. Hence, since $\mathcal L(\cdot)^{1/p}$
	is an equivalent norm to  $\|\cdot \|_{\mathcal{D}^{1,p}(\R^N)}$, we finally get that $u_n\rightarrow u_0$ in  $\mathcal{D}^{1,p}(\R^N)$. Therefore passing to the limit in \eqref{eq:quotient}, we obtain 
	\[S(\gamma)=\frac{\mathcal L(u_0)}{\left(\int_{\mathbb R^N}|u_0|^{p^*}\, dx\right) ^\frac{p}{p^*}},\] namely  $u_0$ (eventually redefining it as $Cu_0$, for some positive constant) is a  weak solution to \eqref{eq:ProbDoublyCritical}.
	
	Now we prove that actually \eqref{eq:weaklimit} holds, concluding indeed the proof. Let $\{u_n\}$ the minimizing sequence such that \eqref{eq:slavin} holds. For every $n$ let us take a sequence of radii $R_n$ such that  
	\begin{equation}\label{bc}
	    \int_{B_{R_n}}H^p(\nabla u_n)-\gamma \frac{|u_n|^p}{H^\circ(x)^p} \, dx=\int_{\mathbb R^N\setminus B_{R_n}}H^p(\nabla u_n)-\gamma \frac{|u_n|^p}{H^\circ(x)^p} \, dx= \frac{\mathcal L(u_n)}{2}
     \end{equation}
	and let us define the rescaled sequence 
	\begin{equation}\label{eq:riun}
		w_n=R_n^{\frac{p-N}{p}}u_n\left (\frac{x}{R_n}\right).
	\end{equation}
	Using assumptions $(h_H)$, \eqref{bc} and \eqref{eq:riun} we deduce that
	\begin{equation}\label{eq:datmec}
		\int_{B_{1}}H^p(\nabla w_n)-\gamma \frac{|w_n|^p}{H^\circ(x)^p} \, dx=\int_{\mathbb R^N\setminus B_{1}}H^p(\nabla w_n)-\gamma \frac{|w_n|^p}{H^\circ(x)^p} \, dx= \frac{\mathcal L(u_n)}{2}
	\end{equation}
	Now we prove the following 
	\begin{lem}\label{lem:tesilem}
		Let  $\{u_n\}$ be a minimizing sequence, weakly converging to zero. Then, for every ball $B_r$ and for every $\varepsilon\in (-r,r)$ there exists $\rho \in (0,\varepsilon)\cup (\varepsilon, 0)$ such that for a subsequence
		\begin{equation}\label{eq:tesilem}
			\text{either}\,\, \int_{B_{r+\rho}}H^{p}(\nabla u_{n})\, dx\rightarrow 0, \qquad \text{or}\,\, \int_{\mathbb R^N \setminus B_{r+\rho}}H^{p}(\nabla u_{n})\, dx\rightarrow 0.
		\end{equation}
	\end{lem}
	\begin{proof}
		By the homogeneity of the quotient in \eqref{eq:quotient} we can assume that the minimizing sequence $\{u_n\}$ is such that $\|u_n\|_{L^{p^*}(\mathbb R^N)}=1$, so that $\mathcal L(u_n)\rightarrow S(\gamma)$. By Ekeland's $\varepsilon-$principle, we can suppose that the minimizing sequence has the Palais-Smale property, that is 
		\begin{equation}\label{eq:palaisSmale}
			\int_{\R^N}H^{p-1}(\nabla u_n) \langle \nabla H(\nabla u_n) , \nabla \varphi \rangle - 
			\frac{\gamma}{H^\circ(x)^p} u_n^{p-1} \varphi \, dx= S(\gamma)\int_{\R^N} u_n^{p^*-1} \varphi \, dx +o(1)\|\varphi\|_{\mathcal{D}^{1,p}(\R^N)},  
		\end{equation}
		for all $\varphi \in \mathcal{D}^{1,p}(\R^N)$
		We have that 
		\begin{equation*}
			\int_r^{r+\varepsilon}\, d\rho\int_{\rho S^{N-1}}H^{p}(\nabla u_n)= \int_{B_{r+\varepsilon}\setminus B_r}H^{p}(\nabla u_n)\end{equation*}
		is bounded. Then we can find $\rho \in (0,\varepsilon)$ such that for infinitely many $n$'s it holds
		\begin{equation*}
			\int_{(r+\rho) S^{N-1}}H^{p}(\nabla u_n)\leq C \int_{B_{r+\varepsilon}\setminus B_r}H^{p}(\nabla u_n),
		\end{equation*}
		for some positive constant $C$ and hence up to redefining the constant  
		\begin{equation*}
			\int_{(r+\rho) S^{N-1}}|\nabla u_n|^p\leq C \int_{B_{r+\varepsilon}\setminus B_r}H^{p}(\nabla u_n).
		\end{equation*}
		Therefore (see \cite[Theorem A.8]{Stru}) since \begin{equation}\label{eq:Struwe}
			W^{1,p}((r+\rho) S^{N-1})\hookrightarrow W^{1-\frac{1}{p},p}((r+\rho) S^{N-1})\hookrightarrow L^p((r+\rho) S^{N-1})\end{equation}
		with both embedding compact, we can assume that a subsequence  converges strongly to some limit, say $u$ in the trace space $W^{1-\frac{1}{p},p}((r+\rho) S^{N-1})$. Using the fact that the trace operator has a continuous embedding from $W^{1,p}(B_{r+\rho})$ into $W^{1-\frac{1}{p},p}((r+\rho) S^{N-1})$, by the weak convergence to zero of $\{u_n\}$, we deduce that indeed $u\equiv 0$.

		Now we show the following

		{\sc Claim:} Let $\Omega\subset \mathbb R^N$ a generic smooth bounded domain. The inverse operator \[(-\Delta_p^H)^{-1}: W^{1-\frac 1p,p}(\partial\Omega)\rightarrow  W^{1,p}(\Omega),\] is continuous. Indeed we consider a succession $g_n\rightarrow g$ in $W^{1-\frac 1p,p}(\Omega)$, and let $u_n,u\in W^{1,p}(\Omega)$ be the solutions to 
		\begin{equation}\label{eq:inverseoperator}
			\begin{cases}
				-\Delta_p^Hu=0 & \text{in}\,\, \Omega
				\\
				u=g & \text{on}\,\, \partial\Omega,   \quad
			\end{cases}
         \quad
         \begin{cases}
				-\Delta_p^Hu_n=0 & \text{in}\,\, \Omega
				\\
				u_n=g_n & \text{on}\,\, \partial\Omega.   \quad
			\end{cases}
   \end{equation} 
		The solution to \eqref{eq:inverseoperator} can be obtained minimizing the functional 
		\[J(u)=\frac{1}{p}\int_{\Omega}H^{p}(\nabla u)\, dx\] on the set $\left\{\{g\}+ W^{1,p}_0(\Omega)\right\}$, $\left\{\{g_n\}+ W^{1,p}_0(\Omega)\right\}$ respectively. Since $(u-g),(u_n-g_n)\in W^{1,p}_0(\Omega)$, integrating by parts \eqref{eq:inverseoperator} and subtracting the equations, we obtain 
		\begin{equation*}
			\begin{split}
			0&=\int_{\Omega}H^{p-1}(\nabla u)\langle \nabla H(\nabla u) , \nabla (u-g) \rangle\, dx\\
			& -\int_{\Omega}H^{p-1}(\nabla u_n)\langle \nabla H(\nabla u_n) , \nabla (u_n-g_n) \rangle\, dx
   \\&=\int_{\Omega}\langle H^{p-1}(\nabla u) \nabla H(\nabla u)-H^{p-1}(\nabla u_n) \nabla H(\nabla u_n) , \nabla u-\nabla u_n\rangle\, dx\\&
   -\int_{\Omega}\langle H^{p-1}(\nabla u) \nabla H(\nabla u) -H^{p-1}(\nabla u_n) \nabla H(\nabla u_n), \nabla g-\nabla g_n \rangle\, dx.
		\end{split}
		\end{equation*}
		We recall (see \cite[Lemma $4.1$]{CSR}) that for $x\in \R^N\setminus \{0\}$, $y\in \R^N$ there exist a constant $C>0$ such that 

\begin{equation}\label{ok}
    \langle H^{p-1}(x)\nabla H(x)-H^{p-1}(y)\nabla H(y),x-y\rangle \geq C(|x|+|y|)^{p-2}|x-y|^2,
\end{equation}
Therefore, by \eqref{ok} and \eqref{zaia} we get
  
		\begin{equation}\label{eq:luglio}
  \begin{split}
			C&\int_{\Omega}\left(|\nabla u|+|\nabla u_n|)^{p-2}\left(|\nabla u-\nabla u_n|\right)^2\right)\, dx\\ &\leq \tilde C_p\int_{\Omega}(|\nabla u|+|\nabla u_n|)^{p-2}|\nabla u-\nabla u_n||\nabla g-\nabla g_n|\, dx \\ &\leq \tilde C_p \left(\int_\Omega(|\nabla u|+|\nabla u_n|)^p\,dx\right)^{\frac{p-1}{p}}\left(\int_\Omega|\nabla g-\nabla g_n|^p\, dx\right)^{\frac1p},
   \end{split}
		\end{equation}
  where in the last inequality we have used the H\"older inequality.

		We recall that by $W^{1-\frac 1p,p}(\partial \Omega)$ we denote the space of  traces $u_{|\partial \Omega}$, namely the set (of equivalence classes) $\left\{\{u\}+ W^{1,p}_0(\Omega),  u\in W^{1,p}(\Omega)\right\}$, endowed with the trace norm 
		\begin{equation}\label{eq:Arnold1}\|u_{|\partial \Omega}\|_{W^{1-\frac 1p,p}(\partial \Omega)}=\inf \{\|v\|_{W^{1,p}(\Omega)}\, :\, u-v \in W^{1,p}_0(\Omega) \}.\end{equation}
		Hence using \eqref{eq:luglio} letting the boundary data  $g_n\rightarrow g$ in the sense of  \eqref{eq:Arnold1} we obtain the claim.
		
		\
		
		Let us define the two auxiliary sequences of $\mathcal{D}^{1,p}(\R^N)$ as follows: 
		\begin{equation*}
			u_{1,n}(x)=
			\begin{cases}
				u_n(x)& \text{if}\,\, x\in B_{r+\rho}\\
				w_{1,n}(x) & \text{if}\,\, x\in B_{r+\varepsilon}\setminus B_{r+\rho}\\
				0 & \text{elsewhere};
			\end{cases}
		\end{equation*}
		and
		\begin{equation}\label{eq:fellfall}
			u_{2,n}(x)=
			\begin{cases}
				0 & \text{if} \,\, x\in B_{r-\varepsilon}\\
				w_{2,n}(x) & \text{if}\,\, x\in B_{r+\rho}\setminus B_{r-\varepsilon}\\
				u_n(x)& \text{elsewhere},
			\end{cases}
		\end{equation}
		where $w_{1,n}$ respectively $w_{2,n}$ denote the solutions to 
		\begin{equation}\label{eq:Arnold2}
			\begin{cases}
				-\Delta_p^Hw_{1,n}=0 & \text{in}\,\,  B_{r+\varepsilon}\setminus B_{r+\rho}
				\\
				w_{1,n}=0& \text{on}\,\, \partial B_{r+\varepsilon}\\
				w_{1,n}=u_n & \text{on}\,\, \partial B_{r+\rho}
			\end{cases}
		\end{equation} 
		respectively
		\begin{equation}\label{eq:Arnold3}
			\begin{cases}
				-\Delta_p^Hw_{2,n}=0 & \text{in}\,\,  B_{r+\rho}\setminus B_{r-\varepsilon}
				\\
				w_{2,n}=u_n & \text{on}\,\, \partial B_{r+\rho}\\
				w_{2,n}=0& \text{on}\,\, \partial B_{r-\varepsilon}.
			\end{cases}
		\end{equation} 
		Since $u_n \rightarrow 0$  on $\partial B_{r+\rho}$   in the $W^{1-\frac 1p,p}$ norm, see \eqref{eq:Struwe}, by  the  above claim we immediately get that both
		\begin{equation}\label{eq:elei1}w_{1,n}\overset{W^{1,p}}{\longrightarrow} 0\quad \text {in }B_{r+\varepsilon}\setminus B_{r+\rho}\qquad \text{and}\qquad w_{2,n}\overset{W^{1,p}}{\longrightarrow} 0\quad \text {in }B_{r+\rho}\setminus B_{r-\varepsilon}.\end{equation}
		Using $u_{1,n}$ as test function in \eqref{eq:palaisSmale} we obtain
		\begin{eqnarray*}
			&&\int_{\R^N}H^{p-1}(\nabla u_n) \langle \nabla H(\nabla u_n) , \nabla u_{1,n} \rangle - 
			\frac{\gamma}{H^\circ(x)^p} u_n^{p-1} u_{1,n} \, dx\\
			&&= S(\gamma)\int_{\R^N} u_n^{p^*-1} u_{1,n} \, dx +o(1)\|u_{1,n}\|_{\mathcal{D}^{1,p}(\R^N)} 
		\end{eqnarray*}
		and  recalling the definition of $u_{1,n}$  and by \eqref{eq:elei1},  we obtain 
		\begin{equation*}
			\int_{B_{r+\rho}}H^{p}(\nabla u_n)  - 
			\frac{\gamma}{H^\circ(x)^p} u_n^{p}  \, dx= S(\gamma)\int_{B_{r+\rho}} u_n^{p^*} \, dx +o(1)= S(\gamma)\int_{B_{r+\rho}} u_{1,n}^{p^*} \, dx +o(1)
		\end{equation*}
		In the same way, using $u_{2,n}$ as test function in \eqref{eq:palaisSmale} we obtain
		\begin{equation*}
			\int_{\mathbb R^N\setminus B_{r+\rho}}H^{p}(\nabla u_n)  - 
			\frac{\gamma}{H^\circ(x)^p} u_n^{p}  \, dx= S(\gamma)\int_{\mathbb R^N\setminus B_{r+\rho}} u_n^{p^*} \, dx +o(1)= S(\gamma)\int_{\mathbb R^N\setminus B_{r+\rho}} u_{2,n}^{p^*} \, dx +o(1).
		\end{equation*}
		Moreover, by definition \eqref{eq:lampedus}, using the two sequences $\{u_{1,n}\}$ and  $\{u_{2,n}\}$  we   infer that actually
		\[{\mathcal L(u_{1,n})=\int_{B_{r+\rho}}H^{p}(\nabla u_n)  - 
			\frac{\gamma}{H^\circ(x)^p} u_n^{p}  \, dx+o(1)}
		\]
		and 
		\[{\mathcal L(u_{2,n})=\int_{\mathbb R^N\setminus B_{r+\rho}}H^{p}(\nabla u_n)  - 
			\frac{\gamma}{H^\circ(x)^p} u_n^{p}  \, dx+o(1)}
		\]
		so that
		\begin{equation}\label{eq:tattoo}
			\mathcal L(u_{n})=\mathcal L(u_{1,n})+\mathcal L(u_{2,n})+o(1)
		\end{equation}
		and $\|u_n\|_{p^*}^{p^*}=\|u_{1,n}\|_{p^*}^{p^*}+\|u_{2,n}\|_{p^*}^{p^*}+ o(1).$ Let us  assume for example, that $\{u_{1,n}\}$ does not converges to zero. Since $\{u_n\}$ is a minimizing sequence we have that (see \eqref{eq:quotient}) 
		\[\mathcal{L}(u_n)=S(\gamma)\|u_n\|^p_{p^*}+o(1)\] and  also that $\mathcal{L}(u_{2,n})\geq S(\gamma)\|u_{2,n}\|^p_{p^*}$ and
		\begin{eqnarray*}
			\frac{\mathcal{L}(u_{1,n})}{\|u_{1,n}\|^p_{p^*}}&=&\frac{\mathcal{L}(u_{n})-\mathcal{L}(u_{2,n})+o(1)}{(\|u_{n}\|^{p^*}_{p^*}-\|u_{2,n}\|^{p^*}_{p^*}+o(1))^\frac{p}{p^*}}\\
			&\leq& S(\gamma)\frac{\mathcal{L}(u_{n})-\mathcal{L}(u_{2,n})+o(1)}{(\mathcal{L}(u_{n})^{\frac{p^*}{p}}-\mathcal{L}(u_{2,n})^{\frac{p^*}{p}}+o(1))^\frac{p}{p^*}}.
		\end{eqnarray*}
		By \eqref{eq:tattoo} we deduce that
		\[\limsup_n \mathcal{L}(u_{2,n})=\limsup_n (\mathcal{L}(u_{n})- \mathcal{L}(u_{1,n})+o(1))< \limsup_n \mathcal{L}(u_{n}),\]
		by some computations we deduce that actually 
		\[
		\limsup_n\frac{\mathcal{L}(u_{1,n})}{\|u_{1,n}\|^p_{p^*}}<S(\gamma),
		\]
		a contradiction with \eqref{eq:quotient} unless $\mathcal{L}(u_{2,n})$ tends to zero. 
		Using Hardy inequality in \eqref{eq:lampedus}, recalling \eqref{eq:fellfall}, 
		we obtain 
		\begin{equation*}
			\begin{split}
			0 \leftarrow \int_{\mathbb R^N}H^{p}(\nabla u_{2,n})\, dx&=\int_{\mathbb R^N \setminus B_{r+\rho}}H^{p}(\nabla u_{n})\, dx+ \int_{B_{r+\rho}\setminus B_{r-\varepsilon}}H^{p}(\nabla w_{2,n})\, dx\\
			&=\int_{\mathbb R^N \setminus B_{r+\rho}}H^{p}(\nabla u_{n})\, dx +o(1).
		\end{split}
		\end{equation*}
		The other case of \eqref{eq:tesilem} can be proved arguing in the same as we have done above, assuming that $\{u_{2,n}\}$ does not converge to zero. This concludes the proof of the lemma.
	\end{proof} 
	Using the invariance of the problem under the scaling 
	$R_n^{{(p-N)}/{p}}u({x}/{R_n})$, the sequence $\{w_n\}$ (see \eqref{eq:riun} and \eqref{eq:datmec}) is still a minimizing sequence bounded in $\mathcal{D}^{1,p}(\R^N)$: hence it admits (up to a subsequence) a weakly convergence sequence. We want to show that the weak limit cannot be zero. We argue by contradiction and we apply Lemma \ref{lem:tesilem} twice choosing $r=1$ and $\varepsilon=\pm 1/4$ respectively. We find the existence of $\rho^+\in (0,1/4)$ and $\rho^-\in (-1/4, 0)$ such that \eqref{eq:tesilem} holds. Using the alternative \eqref{eq:tesilem} together with  \eqref{eq:datmec}, we obtain that 
	\begin{equation}\label{eq:rmol}
		\int_{B_{1+\rho^-}}H^p(\nabla w_n)\, dx \rightarrow 0\quad \text{and}\quad \int_{\mathbb R^N\setminus B_{1+\rho^+}}H^p(\nabla w_n)\, dx \rightarrow 0.
	\end{equation}
	Since $w_n\rightharpoonup 0$ using the strong convergence on compacts $K$ of $w_n \rightarrow 0$ in $L^p(K)$ we deduce that
	\begin{equation}\label{bio}
	 \int_{B(0)_{\frac 32}\setminus B(0)_{\frac12}}\frac{|w_n|^p}{H^\circ(x)^p} \, dx\rightarrow 0.\
  \end{equation}
	Let us take a smooth   cut-off function $\eta$, with $0\leq \eta\leq 1$, such that $\eta\equiv 1$ in $B(0)_{ 5/4}\setminus B(0)_{3/4}$ and $\eta \equiv 0$ in $\mathbb R^N\setminus (B(0)_{ 3/2}\setminus B(0)_{1/2})$. Take in to account Hardy inequality, \eqref{eq:rmol} and \eqref{bio} we deduce that 
	\[\int_{\mathbb R^N}|H^p(\eta \nabla w_n)-H^p(\nabla w_n)|\, dx\rightarrow0,\]
	for $n\rightarrow +\infty$ and therefore we infer that
	\[\int_{\mathbb R^N}H^p(\eta \nabla w_n)\, dx=\int_{\mathbb R^N}H^p(\nabla w_n)\, dx+o(1).\] 
	Then, recalling also \eqref{eq:quotient} we get 
	\begin{equation*}
		S(0)\leq\frac{\displaystyle\int_{\mathbb R^N}H^p(\eta \nabla w_n)\, dx}{\left(\displaystyle \int_{\mathbb R^N}|\eta w_n|^{p^*}\, dx\right) ^\frac{p}{p^*}}\leq \frac{\mathcal L(w_n)+o(1)}{\displaystyle\left(\int_{\mathbb R^N}| w_n|^{p^*}\, dx\right) ^\frac{p}{p^*}+o(1)}.
	\end{equation*}
	Passing to the limit  we obtain 
	\begin{equation}\label{eq:S(0)}
		S(0)\leq S(\gamma).
	\end{equation}
	We claim that \eqref{eq:S(0)} is not possible. Indeed let $u_0$ be a minimizer of  \eqref{eq:quotient} for $\gamma=0$ (\cite{CFR}). Therefore we clearly deduce the following
	\begin{equation}\label{eq:contradbocc}
		\begin{split}
		S(\gamma)&\leq \frac{\displaystyle \int_{\mathbb R^N}H^p(\nabla u_0)-\gamma \frac{|u_0|^p}{H^\circ(x)^p} \, dx}{\displaystyle \left(\int_{\mathbb R^N}|u_0|^{p^*}\, dx\right) ^\frac{p}{p^*}} <\frac{\displaystyle \int_{\mathbb R^N}H^p(\nabla u_0) \, dx}{\displaystyle \left(\int_{\mathbb R^N}|u_0|^{p^*}\, dx\right) ^\frac{p}{p^*}}=S(0).
		\end{split}
	\end{equation}
	Inequality \eqref{eq:contradbocc} gives the desired contradiction, concluding the proof. 
\end{proof}

\end{document}